\documentclass{article}
\usepackage{amsmath,amssymb,amsthm}
\usepackage{authblk}
\usepackage{tikz}
\usepackage{esint}
\usepackage{hyperref}
\usepackage[capitalize,nameinlink,noabbrev]{cleveref}

\renewcommand*{\d}{\mathop{\kern0pt\mathrm{d}}\!{}}
\renewcommand{\div}{\operatorname{div}}
\newcommand{\dt}{\d t}
\newcommand{\dx}{\d x}
\newcommand{\dy}{\d y}
\newcommand{\cross}{\times}
\newcommand{\grad}{\nabla}
\newcommand{\La}{\Delta}
\newcommand{\nor}[1]{\left\| #1 \right\|}
\newcommand{\R}{\mathbb R}
\newcommand{\pthf}[2]{\left(\frac{#1}{#2}\right)}
\newcommand{\BMO}{\textit{BMO}}
\newcommand{\curl}{\operatorname{curl}}
\newcommand{\e}{\mathrm e}
\newcommand{\cM}{\mathcal M}

\newtheorem{theorem}{Theorem}[section]%
\newtheorem{proposition}[theorem]{Proposition}%
\newtheorem{lemma}[theorem]{Lemma}%
\newtheorem{corollary}[theorem]{Corollary}%

\theoremstyle{remark}%
\newtheorem{remark}{Remark}%

\theoremstyle{definition}%
\newtheorem{definition}[theorem]{Definition}%

\title{Conditional Liouville Theorems for the Navier--Stokes Equations}
\date{\today}
\author[1]{Matei P. Coiculescu}
\author[2]{Jincheng Yang}
\affil[1]{Department of Mathematics, Princeton University, Princeton, NJ 08544\newline {\it coiculescu@princeton.edu}}
\affil[2]{Department of Applied Mathematics and Statistics, Johns Hopkins University, Baltimore, MD 21211 {\it jincheng@jhu.edu}}

\begin{document}
\maketitle
\begin{abstract}
We present a novel approach to the Liouville problem for the stationary Navier--Stokes equations. As an application of our method, we prove conditional Liouville theorems with assumptions on the antiderivative of the velocity that represent substantial improvements on what was heretofore known.
\end{abstract}

\section{Introduction}

The Navier--Stokes equations are a model for the motion of viscous incompressible fluid. The viscosity of the fluid causes an energy dissipation that is typical of parabolic partial differential equations (like the heat equation). In this sense, we consider the Navier--Stokes equations as a gradient flow dissipating the energy of a fluid until it approaches a limiting end state, which we expect should be a stationary solution of the equations. Our quotidian observation of fluids further suggests that the end-state should be a fluid at rest, i.e. with Eulerian velocity $u\equiv 0$.  Rigorous mathematical proof of this expectation has so far been elusive, but work on the stationary Navier--Stokes equations remains an important part of the field of mathematical fluid mechanics. The present note thus concerns the stationary Navier--Stokes equations in $\R^3$:
\begin{align}
    \label{eqn: SNS}
    -\Delta u + u\cdot\nabla u + \nabla p =0,\\
    \textrm{div } u =0,
\end{align}
where $u$ denotes the velocity field of the fluid and $p$ denotes the scalar pressure field of the fluid. Weak solutions with $u\in \mathrm{L}^3_{loc}(\R^3)$ of Equation \eqref{eqn: SNS} are smooth, see Theorem X.1.1 in Galdi's monograph \cite{G}. For this reason, we may always assume that $u$ and $p$ are smooth functions on $\R^3$. We shall also assume that $u$ satisfies the natural vanishing condition at infinity:
\begin{align}
    \label{eqn: vanish}
    \lim _{|x| \to +\infty} u (x) = 0.
\end{align}
The question remains open, however, whether there exist any nontrivial solutions (i.e. solutions besides $u\equiv 0$). In particular, do there exist nontrivial solutions of Equation \eqref{eqn: SNS} with finite Dirichlet energy, i.e. with
\begin{align}
    \label{eqn: D}
    \int_{\R^3} |\nabla u|^2 \dx <\infty?
\end{align}
Solutions with finite Dirichlet energy and satisfying the vanishing condition Equation \eqref{eqn: vanish} are often called $D$-solutions, and the question whether nontrivial $D$-solutions exist is called the Liouville problem for the three-dimensional Navier--Stokes equations. We note here that for the stationary Navier--Stokes equations in $\R^n$ with $n\neq 3$, there are no nontrivial $D$-solutions.

We now survey some of the progress done towards resolving this question. Observe that the vanishing condition in Equation \eqref{eqn: vanish}, the assumption of finite Dirichlet energy as in Equation \eqref{eqn: D}, and the Sobolev embedding in $\R^3$ together imply that a $D$-solution $u$ lies in $\mathrm{L}^6(\R^3)$. Galdi proves in \cite{G} that if one assumes more decay at infinity, in particular if one assumes that $u\in \mathrm{L}^{9/2}(\R^3)$, then $u$ is trivial. Chae and Wolf prove a logarithmically improved version of the $\mathrm{L}^{9/2}(\R^3)$ criterion in \cite{CW}.

Another direction for resolving the Liouville problem seems to have begun with the work of Seregin, Silvestre, Sverak, and Zlatos in \cite{SSSZ}, where the authors consider linear elliptic and parabolic systems with a divergence-free drift term. In particular, they prove a Liouville theorem for the Laplace equation with drift when the drift term is divergence-free and in the scale-invariant space $\BMO^{-1}$. We recall that a divergence-free drift vector field $b_i$ can be written as the divergence of some tensor $T_{ij}$, and the assumption that $b\in \BMO^{-1}$ is equivalent to assuming that $T\in \BMO$.  The authors of \cite{SSSZ} suggest that proving deep properties, like a Liouville theorem for an elliptic system, requires scale-invariant information: one should not be able to break the scaling. The assumption that $u\in \BMO^{-1}$ lies exactly in this category of hypotheses for the stationary Navier--Stokes system as well. Indeed, Seregin proves in \cite{S1} that if a solution $u$ of Equation \eqref{eqn: SNS} is in $\mathrm{L}^6(\R^3)\cap \BMO^{-1}(\R^3)$, then $u$ is trivial (we note that the finite Dirichlet energy assumption is not directly used in Seregin's proof).  

Seregin proves the first Liouville-type result with a quantitative version of the $\BMO^{-1}$ condition in \cite{S2}. We denote balls of radius $R$ centered at the origin as $B_R(0)$ or also as $B_R$ when there is no confusion. Suppose $u= \textrm{curl } \psi$ is a smooth solution of Equation \eqref{eqn: SNS}. Seregin proves in \cite{S2} that with the assumption
\begin{align*}
     \left( \fint _{B _R} |\psi - (\psi) _{B _R}| ^s \dx \right) ^\frac1s \le C R ^{\alpha (s)},
\end{align*}
for some $s>3$ and constant $C>0$ with $\alpha(s) = \frac{s-3}{6(s-1)}$, the smooth solution $u$ of Equation \eqref{eqn: SNS} is trivial. In their article \cite{CW2}, Chae and Wolf prove that if $u = \div T$ for some tensor $T$ with 
\begin{align*}
    \left( \fint _{B _R} |T - (T) _{B _R}| ^s \dx \right) ^\frac1s \le C R ^{\alpha (s)}
\end{align*}
for some $s>3$ and constant $C>0$ with $\alpha (s) = \min \{\frac13 - \frac1s, \frac16\}$, then $u$ is trivial. The recent work \cite{BY} of Bang and Yang provides a logarithmic improvement of the Chae--Wolf result. Lastly, we recommend Section 6 in the survey article \cite{SS} as another short summary of what has been done in this direction for the Liouville problem. The references in \cite{SS} are exhaustive up to the year of its publication.

The goal of the present work is to present a novel approach to the Liouville problem of the stationary Navier--Stokes equations. The main difference to our approach, in comparison with previous methods of attack, is that we perform local estimates on domains called \textit{capsules} that we allow to vary in size from point to point, depending on the local Dirichlet energy and local average velocity. The main point is that if the velocity does not \textit{stretch} too much, we can prove a Liouville theorem. The best elucidation of our new approach will be in proving the following theorems:

\begin{theorem}
\label{FIRSTTHM}
    Suppose $u \in C ^\infty (\R^3)$ is a $D$-solution of the stationary Navier--Stokes equations. Let $s\geq 1$ be arbitrary. If $u = \curl \psi$ and satisfies
    \begin{align}
        \label{eqn: first-thm}
        \left(\fint _{B _R (x _0)} |\psi - (\psi) _{B _R (x _0)}|^s \dx\right)^{\frac{1}{s}} \le C R ^{\alpha}, \qquad \forall R > 1, \forall x _0 \in \R ^3
    \end{align}
    for some constants $C>0$ and $0<\alpha<1$, then $u \in \mathrm{L}^{p, \infty}$ for any $p > \frac4{1 - \alpha}$. Moreover, if $\alpha < \frac19$, then $u \equiv 0$.
\end{theorem}
\begin{theorem}
\label{SECONDTHM}
    Suppose $u\in C^\infty(\R^3)$ is a $D$-solution of the stationary Navier--Stokes equations. If $u$ satisfies:
    \begin{align}
        \label{eqn: second-thm}
        \int _{x _0} ^x u \cdot \d \ell \le C|x-x _0|^\beta, \qquad \forall x _0, x\in \R^3,
    \end{align}
    for some constants $C>0$ and $0<\beta<1$, then $u \in \mathrm{L}^{p, \infty}$ for any $p > \frac{4 - 34 \beta / 29}{1 - \beta}$. Moreover, if $\beta<\frac{29}{193}$, then $u \equiv 0$.
\end{theorem}

Our \cref{FIRSTTHM} improves the result of Seregin by allowing $\mathrm{L}^s$ control on the mean oscillation for $1 \leq s \leq 3$ as well. In addition, our hypothesis is weaker then Seregin's for $3 < s < 7$. Likewise, our \cref{FIRSTTHM} improves the Chae--Wolf result: our hypothesis is weaker than their hypothesis when $3< s< 9/2$. We also remark that our threshold $\alpha=1/9$ is uniform across the exponents. 

Our \cref{SECONDTHM} shows that $u\equiv 0$ when we control the line integral of $u$. The integral $\int _{x_0} ^{x }$ means integration on a straight line segment connecting $x_0$ and $x$. For $\beta < 1$, the assumption in Equation \eqref{eqn: second-thm} holds automatically if $|x - x _0|$ is small since $u$ is bounded. Therefore, Equation \eqref{eqn: second-thm} is also a type of growth condition.

We remark that the hypotheses of our results encode quantitative information about the antiderivative of solutions of the stationary Navier--Stokes equations, in line with the previous works outlined above. Indeed, if $u$ is a potential flow, then the left-hand side of \eqref{eqn: second-thm} defines its potential function $\phi (x)$ with $u = \nabla \phi$. However, in this scenario the triviality of $u$ follows directly from the Liouville theorem for harmonic functions: there are no non-constant entire harmonic functions with sublinear growth. The stream function $\psi$ can be defined by a similar line integral of \eqref{eqn: second-thm} with $-u ^\perp$ replacing $u$ for 2D incompressible flows. In this way, the stream function can also be understood as a type of antiderivative. 

The rest of this paper lies in setting up our \textit{capsule} approach and applying it to prove our theorems. We also prove a more general conditional Liouville theorem with a hypothesis on the capsules alone, quantifying the \textit{stretching} condition mentioned before.

In the next section, we prove local estimates for the Navier--Stokes system. In the following section, we use the \textit{capsules} to get global control on the shearing of the solution of Equation \eqref{eqn: SNS}. Then we prove some conditional Liouville theorems as consequences of the capsule method. We perform some technical estimates on the kernel of the Poisson equation with drift in the appendix. A short remark on notation: an inequality denoted with $\lesssim$ will always mean an inequality up to an absolute computable constant. For scalars or vectors $f$ and $g$, $f \approx g$ means $|f - g| \lesssim \min \{|f|, |g|\}$. Lastly, we denote the positive part of a real number $\alpha$ by $\alpha_+:= \max(0, \alpha)$.

\section{Local Estimates for the Navier--Stokes System}

We define a \textit{capsule} $\mathcal C$ of radius $R$, half-length $L \ge R$ in the direction $e \in \mathbb S ^2$, and center $x$ to be defined by 
\begin{align}
    \label{eqn: capsule}
    \mathcal C_{R,L,e}(x):= \bigcup _{t \in [-L + R, L - R]} (t e + B _R(x)).
\end{align}
Here $B _R(x) \subset \R^3$ is a ball of radius $R$ centered at $x$, and we shall abbreviate $B _R = B _R (0)$. Note that $\mathcal C_{R,R,e}(x)= B _R(x)$ for any $e$. Given a capsule $\mathcal{C}_{R,L,e}(x)$ and any $\lambda>0$, we denote $\lambda \mathcal C_{R,L,e}(x):= \mathcal{C}_{\lambda R, \lambda L, e}(x)$. Where there is no confusion, we may omit dependencies on $R,L,e,$ or $x$. Occasionally, we may denote the capsule by $\mathcal{C} _x$ to highlight the dependency on the center point $x$.

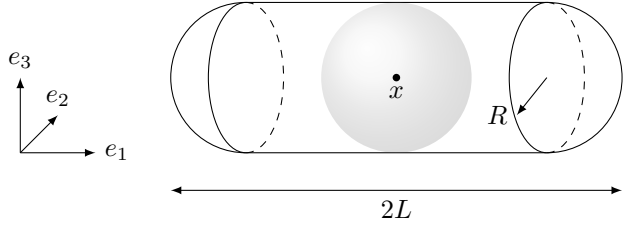
\begin{figure}
    \centering
    \begin{tikzpicture}
    \def\R{1} %
    \def\L{2} %

    \draw[] (-\L, \R) -- (\L, \R);
    \draw[] (-\L, -\R) -- (\L, -\R);

    \draw[] (-\L, \R) arc[start angle=90, end angle=270, radius=\R];
    \draw[] (\L, \R) arc[start angle=90, end angle=-90, radius=\R];
    \draw[thin] (-\L, \R) arc (90:270:\R/2 and \R);
    \draw[dashed] (-\L, -\R) arc (-90:90:\R/2 and \R);
    \draw[thin] (\L, \R) arc (90:270:\R/2 and \R);
    \draw[dashed] (\L, -\R) arc (-90:90:\R/2 and \R);
    \shade[ball color = gray!40, opacity = 0.2] (0,0) circle (\R);
    \fill (0, 0) circle (.05) node [anchor=north] {$x$};

    \draw[->, -latex] (\L, 0) -- (\L - \R/2 + .1, -\R/2) node [anchor=east] {$R$};

    \draw[->, -latex] (-\L-3, -1) -- (-\L-2, -1) node[right] {$e_1$};
    \draw[->, -latex] (-\L-3, -1) -- (-\L-2.5, -0.5) node[above] {$e_2$};
    \draw[->, -latex] (-\L-3, -1) -- (-\L-3, 0) node[above] {$e_3$};
    \draw[<->, latex-latex] (-\L-\R, -\R-.5) -- (\L+\R, -\R-.5) node[midway, anchor=north] {$2 L$};
    \end{tikzpicture}

    \caption{Illustration of the capsule $\mathcal C _{R,L,e _1} (x)$}
    \label{fig}
\end{figure}

Take $\mathcal C = \mathcal C _{R, L, e _1} (0)$ to be a capsule centered at the origin, pointing in the direction of $e _1$. In this section, suppose $u$ is a smooth solution to the stationary Navier--Stokes equations in a dilated capsule $2 \mathcal C$. %
We further assume there exist $U, \Xi > 0$, such that 
\begin{align}
    \label{eqn: U-and-Xi}
    \fint _{B _R} u \dx = U e _1, \qquad \fint _{2 \mathcal C} |\grad u| ^2 \dx = \Xi ^{2}.
\end{align}
We shall now derive a local $\mathrm{L}^\infty$ bound in $\mathcal C$ on $u - U e _1$. Let $V$ denote the volume of the capsule $\mathcal{C}$, or $V= |\mathcal{C}|$. Note that $V \approx L R ^2$. 

\begin{lemma}
    \label{lem: oscillation-estimate}
    If $u$ is a smooth solution to the stationary Navier--Stokes equations satisfying the assumptions in Equation \eqref{eqn: U-and-Xi}, then 
    \begin{align}
    \label{eqn: Linfty-bound}
        \nor{u - U e _1} _{\mathrm{L}^\infty (\mathcal C)} \lesssim \Xi R ^\frac12 V ^\frac12 \left(\frac {UR}L + \frac1{R} + L \Xi \right) (\Xi V ^\frac23 + 1) + L \Xi.
    \end{align}
\end{lemma}

The proof of this lemma will occupy the following three subsections.

\subsection{Control of the Average Velocity}
Denote 
\begin{align*}
    \bar u (t) := \fint _{B _R (t e _1)} u \dx.   
\end{align*}
Then $\bar u (0) = U e _1$. In this subsection we show that $|\bar u (t) - U e _1| \lesssim L \Xi$ provided that $|t| \le 2 (L - R)$. 

\begin{lemma}
    \label{lem: mean-velocity}
    For every $|t| \le 2 (L - R)$, we have  
    \begin{align*}
        |\bar u (t) - U e _1| \lesssim L \left(\fint _{\mathcal C _{2 R, 2 L, e _1}} |\nabla u| ^2 \d x\right) ^{\frac12}.
    \end{align*}
\end{lemma}

To do this, we require the following short lemma.
\begin{lemma}\label{SILLYLEMMA}
Let $f(x)$ be a nonnegative function. Then for every $l > R$, we have
$$\frac{\sqrt3 - 1}{2}R \int _{\mathcal C _{R / 2, l + R / 2, e _1}} f(x) \dx \le \int _{-l} ^{l} \int_{B _R} f(x+te_1) \dx \dt \le 2R \int_{\mathcal C _{R, l + R, e _1}} f(x)\dx.$$
\end{lemma}
\begin{proof}
We first rewrite:
\begin{equation}
    \label{BASICCHANGEVAR}
    \int _{-l} ^{l} \int_{B _R} f(x+te_1) \dx \dt = \int_{-l} ^{l} \int_{B _R (t e _1)} f(x) \dx \dt.
\end{equation}
Let $\mathbf1 _A (x)$ be the characteristic function of a set $A$. We rewrite
$$
    \int _{-l} ^{l}  \int_{B _R (t e _1)} f(x) \dx\dt = \int _{-l} ^{l} \int_{\R ^3} f(x) \mathbf1 _{B _R (t e _1)} (x) \dx\dt.
$$
Using the Fubini--Tonelli Theorem we have:
\begin{equation}\label{STEP1}
    \int _{-l} ^{l}  \int_{B _R (t e _1)} f(x) \dx\dt = \int_{\R ^3} f(x) \left(\int _{-l} ^{l} \mathbf1 _{B _R (t e_1)} (x) \dt\right)\dx.
\end{equation}

Now observe that $x\in B_R(te_1)$ implies $|x-te_1| < R$, 
or equivalently
$$(t-x_1)^2+x_2^2+x_3^2 < R^2,$$
where
$x= (x_1,x_2,x_3)$. We thus have that $t$ necessarily lies in the interval 
$$t\in \left(x_1-\sqrt{R^2-(x_2^2+x_3^2)}, x_1 + \sqrt{R^2-(x_2^2+x_3^2)}\right).$$
The integral of the characteristic function is thus equal to 
\begin{equation}\begin{aligned}
    \label{eqn: integral-indicator}
    \int _{-l} ^{l} \mathbf1 _{B_R+te_1}(x) \dt &= \max \bigg\{\min \left\{l, x _1 + \sqrt{R^2-(x_2^2+x_3^2)}\right\} \\
    & \qquad \qquad - \max\left\{-l, x_1-\sqrt{R^2-(x_2^2+x_3^2)}\right\}, 0\bigg\} \\
    & = \min \bigg\{
        2 \sqrt{R ^2 - (x _2 ^2 + x _3 ^2)}, \\
        & \qquad \qquad \left(\sqrt{R ^2 - (x _2 ^2 + x _3 ^2)} + l - |x _1|\right) _+
    \bigg\}.
\end{aligned}\end{equation}

Let us compute \eqref{eqn: integral-indicator}. On the one hand, if $x \in \mathcal C _{R / 2, l + R / 2, e _1}$, which is equivalent to $(|x _1| - l) _+ ^2 + x _2 ^2 + x _3 ^2 < (R / 2) ^2$, then
\begin{align*}
    \sqrt{R^2-(x_2^2+x_3^2)} > \sqrt{R ^2 - (R / 2) ^2} = \sqrt 3 R / 2.
\end{align*}
Therefore
\begin{align*}
    \int _{-l} ^{l} \mathbf1 _{B_R+te_1}(x) \dt \ge \min \left\{ \sqrt 3 R, \sqrt 3 R / 2 - R / 2 \right\} = \frac{\sqrt3-1}{2} R.
\end{align*}
This is true for any $x \in \mathcal C _{R/2, l + R/2}$, therefore we obtain a lower bound on the integral in Equation \eqref{eqn: integral-indicator}:
\begin{align*}
    \int _{-l} ^{l} \mathbf1 _{B_R+te_1}(x) \dt & \ge \frac{\sqrt3-1}{2} R \cdot \mathbf1 _{\mathcal C _{R/2, l + R/2}} (x).
\end{align*}
On the other hand, for the integral \eqref{eqn: integral-indicator} to be positive, one must have 
\begin{align*}
    \sqrt{R ^2 - (x _2 ^2 + x _3 ^2)} + l - |x _1| > 0,
\end{align*}
equivalent to $(|x _1| - l) _+ ^2 + x _2 ^2 + x _3 ^2 < R ^2$, or $x \in \mathcal C _{R, l + R, e _1}$. Moreover, in this situation, the integral is always smaller than 
\begin{align*}
    \int _{-l} ^{l} \mathbf1 _{B_R+te_1}(x) \d t \le 2 \sqrt{R^2-(x_2^2+x_3^2)} \le 2 R,
\end{align*}
so we obtain an upper bound on the integral in Equation \eqref{eqn: integral-indicator}:
\begin{align*}
    \int _{-l} ^{l} \mathbf1 _{B_R+te_1}(x) \d t \le 2 R \cdot \mathbf1 _{\mathcal C _{R, l + R, e _1}} (x).
\end{align*}
Combine the lower and upper bound, the conclusion of the lemma is immediate from Equation \eqref{STEP1}.
\end{proof}

Using this auxiliary Lemma, we now prove \cref{lem: mean-velocity}.

\begin{proof}[Proof of \cref{lem: mean-velocity}]
    For any $t _1, t _2 \in [-2 L + 2 R, 2 L - 2 R]$, we have, using Equation \eqref{BASICCHANGEVAR} and \cref{SILLYLEMMA} with $l = 2 L - 2 R$ that
    \begin{align}
        \notag
        |\bar u (t _2) - \bar u (t _1)| &= \left|\fint _{B _R} u (x + t _2 e _1) - u (x + t _1 e _1) \dx\right| \\
        \notag
        &= \left|\fint _{B _R} \int _{t _1} ^{t _2} \partial _{x _1} u (x + t e _1) \d t \dx\right| \\
        \notag
        &\le \int _{-2 L + 2 R} ^{2 L - 2 R} \fint _{B _R} |\partial _{x _1} u (x + t e _1)| \dx \d t \\
        \label{eqn: osc-mean-velocity}
        &\le \frac{2 R}{|B_R|}\int_{2 \mathcal{C}}|\nabla u|\dx \le \frac{16 RV}{|B_R|} \left(\fint_{2\mathcal{ C}}|\nabla u|^2 \dx\right)^{1/2}\lesssim L\Xi.
    \end{align}
    In particular, $|\bar u (t) - U e _1| = |\bar u (t) - \bar u (0)| \lesssim L \Xi$ for every $t \in [-2 L + 2 R, 2 L - 2 R]$.    
\end{proof}

\subsection{Velocity Decomposition}

In the next subsection, to avoid the typical difficulties that arise when analyzing the nonlocal pressure field, we shall work instead on the vorticity equation, which does not explicitly depend on the pressure. As a preparation, we now introduce the velocity decomposition.

Fix a non-negative cut-off function $\phi \in C _c ^\infty (2 \mathcal C)$ with $\phi \equiv 1$ in $\frac32 \mathcal C$ and satisfying the following pointwise bounds for some constant $C>0$: 
\begin{align}
        \label{eqn: grad-phi}
        |\partial _{x _1} \phi| & \le \frac CL, &
        |\grad \phi| & \le \frac CR, &
        |\grad ^2 \phi| &\le \frac C{R ^2}.
    \end{align}
    We decompose $u = U e _1 + v + h$, where 
\begin{align*}
    v = \curl (-\La) ^{-1} (\phi \omega)
\end{align*}
and $h = u - U e _1 - v$. Note that $v$ and $h$ are both divergence-free. The vector field $v$ represents the rotational part of the velocity because $\curl v = \omega$ in $\frac{3}{2}\mathcal{C}$. The vector field $h$ is the potential part of the velocity because a simple computation reveals that $h$ is harmonic in $\frac32 \mathcal C$: $\La h = \La u - \La v = - \curl \omega + \curl (\phi \omega) = -\curl ((1 - \phi) \omega)$, which is zero on the set $\{\phi = 1\}$.

Note that, since $\phi\omega$ is a bounded, smooth, compactly supported function, and since the Biot-Savart kernel decays at infinity, we have
\begin{align*}
    \lim_{|x|\to \infty} v(x) =0.
\end{align*}
Moreover, by the Sobolev embedding and the Calder\'{o}n--Zygmund estimate for singular integral operators, we have
\begin{align}
    \label{eqn: L6-H1-bound-v}
    \nor{v} _{\mathrm{L}^6 (\R ^3)} \lesssim \nor{\nabla v} _{\mathrm{L}^2 (\R ^3)} \lesssim \nor{\omega} _{\mathrm{L}^2 (2 \mathcal C)} \lesssim \Xi V ^\frac12.
\end{align}
By the same argument as in the previous subsection, we have
\begin{align*}
    |\bar v (t _2) - \bar v (t _1)| \lesssim  L \Xi, \qquad \forall t _1, t _2 \in [-2 L + 2 R, 2 L - 2 R].
\end{align*}
We have bounded the oscillation of $\bar v$. To get a bound for its size we use the facts that $V\lesssim LR^2$, $\|v\|_{\mathrm{L}^6(\R^3)}\lesssim \Xi V^{1/2}$, and \cref{SILLYLEMMA} to control the average:
\begin{align*}
    \left|
        \fint _{-2 L + 2 R} ^{2 L - 2 R} \bar v (t) \d t
    \right| \le \frac {R}{L |B _R|}
    \left|
        \int _{2\mathcal C} |v (x)| \dx
    \right| \lesssim V ^{-1} \nor{v} _{\mathrm{L}^6 (2 \mathcal C)} V ^\frac56 \le \Xi V ^\frac13.
\end{align*}
Since $L \ge R$, we know $V ^\frac13 \lesssim L$, so 
\begin{align*}
    |\bar v (t)| \lesssim  L \Xi + \Xi V ^\frac13 \lesssim L \Xi, \qquad \forall t \in [-2 L + 2 R, 2 L - 2 R].
\end{align*}
Since $|\bar u (t) - U e _1| \lesssim L \Xi$, we conclude 
\begin{align*}
    |\bar h (t)| \lesssim  L \Xi, \qquad \forall t \in [-2 L + 2 R, 2 L - 2 R].
\end{align*}
In other words, $\bar v$ and $\bar h$ both enjoy the same bound for $\bar u - U e _1$ that we obtained from the previous subsection.

Since $h$ is harmonic inside $\frac32 \mathcal C$, we have by the mean-value-property that $\overline{h}(t)= h(t e _1)$ for all applicable $t$, whence it follows that
\begin{equation}
\label{CENTERLINEBOUND}
    \sup_{t\in [-2 L + 2 R, 2L - 2R]} |h(t e_1)| \lesssim  L \Xi.
\end{equation}
Moreover, from Equation \eqref{eqn: L6-H1-bound-v} we also have 
\begin{align}
    \label{eqn: H1-h}
    \nor{\nabla h} _{\mathrm{L}^2 (2 \mathcal C)} \le \nor{\nabla u} _{\mathrm{L}^2 (2 \mathcal C)} + \nor{\nabla v} _{\mathrm{L}^2 (2 \mathcal C)} \lesssim \Xi V ^\frac12.
\end{align}
In the following lemma, we show that since $h$ is harmonic, we can in fact extend the same bound to the interior of $\frac{3}{2}\mathcal{C}$.
\begin{lemma}
\label{HARMONICBOUND}
Let $h$ be a harmonic function in a capsule $\frac32 \mathcal{C}$ satisfying Equations \eqref{CENTERLINEBOUND} and \eqref{eqn: H1-h}. Then $\|h\|_{\mathrm{L}^\infty(\frac43\mathcal{C})}\lesssim L\Xi$.
\end{lemma}
\begin{proof}
Take $x \in B_{\frac43R}$, $|t| \le \frac43 (L - R)$, then $x + t e _1 \in \frac43 \mathcal C$.
By the mean-value property for $\nabla h$ in $\frac32 \mathcal C$, we have
$$\nabla h(x+te_1) = \fint_{B_{\frac16 R}} \nabla h(x+te_1 +y) \d y,$$
which immediately implies that
\begin{align*}|\nabla h(x+te_1)| \lesssim\frac{1}{|B_{\frac16 R}|} \int_{2\mathcal{C}} |\nabla h| \dx
\lesssim \frac{V}{|B_R|} \left(\fint_{2\mathcal{C}} |\nabla h|^2\right)^{1/2}\lesssim \frac{L}{R}\Xi.
\end{align*}
This is true for any $t, x$, so $h$ is $\frac{C L \Xi}R$-Lipschitz in $\frac43 \mathcal C$.
We now have
$$|h(x+te_1) - h(te_1)|\lesssim |x| \frac{L}{R}\Xi \lesssim L\Xi.$$
Together with Equation \eqref{CENTERLINEBOUND}, we know $|h (x + t e _1)| \lesssim L \Xi$. Since $t$ and $x$ were arbitrary, the claim is proven.
\end{proof}
Since $h$ is bounded in $\frac43 \mathcal{C}$ with bound given by $L\Xi$, we certainly have 
\begin{align*}
    \nor{h} _{\mathrm{L}^6 (\frac32 \mathcal C)} \lesssim L \Xi V ^\frac16 + \Xi V ^\frac12 \lesssim L \Xi V ^\frac16.
\end{align*}
Here we used $V ^\frac13 \lesssim L$. Combined with the $\mathrm{L}^6$ bound on $v$ in Equation \eqref{eqn: L6-H1-bound-v}, we have obtained an $\mathrm{L}^6$ bound for $\zeta := u - U e _1 = v + h$: 
\begin{align}
\label{eqn: L6ESTIMATE}
    \nor{\zeta} _{\mathrm{L}^6 (\frac32 \mathcal C)} \lesssim  L \Xi V ^\frac16 + \Xi V ^\frac12 \lesssim  L \Xi V ^\frac16.
\end{align}

\subsection{The Vorticity Equation}

The vector field $\zeta = u - U e _1$ satisfies the stationary Navier--Stokes equations with constant background drift $b = U e _1$:
\begin{align}
\label{eqn: EQNDRIFT}
    b \cdot \grad \zeta + \zeta \cdot \grad \zeta + \grad p = \La \zeta.
\end{align}
We shall use the Bernoulli head pressure (in the moving frame) $H = \frac{|\zeta|^2}2 + p$ and write Equation \eqref{eqn: EQNDRIFT} in the following form using vorticity $\omega = \curl u = \curl \zeta$:
\begin{align}
\label{eqn: EQNDRIFT1}
    b \cdot \grad \zeta + \omega \cross \zeta + \grad H = \La \zeta.
\end{align}
Taking the curl of Equation \eqref{eqn: EQNDRIFT1} yields
\begin{align*}
    b \cdot \grad \omega + \curl (\omega \cross \zeta) = \La \omega. 
\end{align*}
We use the local estimates proven for the Poisson equation in \cref{lem: local-estimate} of \cref{app} (we recognize that $\textrm{curl}(\omega \times \zeta) = \textrm{div } T$, for the tensor $T_{ij}= \epsilon_{ijk}(\omega\times\zeta)_k$ and where $\epsilon_{ijk}$ is the antisymmetric symbol). We then have for all $1<q<3$ and $1/r=1/q-1/3$:
\begin{align}
    \nor{\omega} _{\mathrm{L}^r (\frac43\mathcal C)} &\lesssim 
        \nor{\omega \otimes \zeta} _{\mathrm{L}^q (\frac32 \mathcal C)} + \left(\frac {UR}L + \frac1R\right) \nor {\omega} _{\mathrm{L}^q (\frac32 \mathcal C)}.
    \label{eqn: APPLYOSEEN}
\end{align}
However, we already know that (also using Equation \eqref{eqn: L6ESTIMATE})
\begin{align*}
    \nor{\omega} _{\mathrm{L}^\frac32 (\frac32 \mathcal C)} &\le \Xi V ^\frac23, &
    \nor{\omega \otimes \zeta} _{\mathrm{L}^\frac32 (\frac32 \mathcal C)} &\le \nor{\omega} _{\mathrm{L}^2 (\frac32 \mathcal C)} \nor{\zeta} _{\mathrm{L}^6 (\frac32 \mathcal C)} \le L \Xi ^2 V ^\frac23 .
\end{align*}
Setting $q=\frac32$ in the estimate from Equation \eqref{eqn: APPLYOSEEN}, we get
\begin{align*}
    \nor{\omega} _{\mathrm{L}^3 (\frac43 \mathcal C)} \lesssim \left(\frac {UR}L + \frac1{R} + L \Xi \right) \Xi V ^\frac23.
\end{align*}

Vorticity in $\mathrm{L}^3$ only guarantees that the velocity lies in $\BMO$, not $\mathrm{L}^\infty$. To get an $\mathrm{L}^\infty$ bound for $\zeta$, we therefore need to refine our estimates. We first note that 
\begin{align*}
    \nor{\omega \otimes \zeta} _{\mathrm{L}^2 (\frac43 \mathcal C)} &\le \nor{\omega \otimes v} _{\mathrm{L}^2 (\frac43 \mathcal C)} + \nor{\omega \otimes h} _{\mathrm{L}^2 (\frac43 \mathcal C)} \\
    & \le \nor{v} _{\mathrm{L}^6 (\frac43 \mathcal C)} \nor{\omega} _{\mathrm{L}^3 (\frac43 \mathcal C)} + \nor{\omega} _{\mathrm{L}^2 (\frac43 \mathcal C)} \nor{h} _{\mathrm{L}^\infty (\frac43 \mathcal C)} \\
    & \lesssim \Xi V ^\frac12 \left[ 
        \left(\frac {UR}L + \frac1{R} + L \Xi \right) \Xi V ^\frac23 + L \Xi
    \right].
\end{align*}
Here we used Equation \eqref{eqn: L6-H1-bound-v} and \cref{HARMONICBOUND}. Setting $q = 2$ in the estimate from Equation \eqref{eqn: APPLYOSEEN} and recalling that $\nor{\omega} _{\mathrm{L}^2} \lesssim \Xi V ^\frac12$, we get 
\begin{align*}
    \nor{\omega} _{\mathrm{L}^6 (\frac54 \mathcal C)} & \lesssim \Xi V ^\frac12 \left[ 
        \left(\frac {UR}L + \frac1{R} + L \Xi \right) \Xi V ^\frac23 + L \Xi
    \right] + \left(\frac {UR}L + \frac1{R}\right) \Xi V ^\frac12 \\
    & = \Xi V ^\frac12 \left(\frac {UR}L + \frac1{R} + L \Xi \right) (\Xi V ^\frac23 + 1). 
\end{align*}
Now, if we fix another cut-off function $\phi _1$ supported in $\frac54 \mathcal C$ and $\phi _1 \equiv 1$ in $\frac65 \mathcal C$ with derivative bounds as in Equation \eqref{eqn: grad-phi}, then $v _1 := \curl (-\Delta) ^{-1} (\phi _1 \omega)$ satisfies
\begin{align*}
    |v _1 (x)| = C \left|
        \int _{\R ^3} \frac{(\phi _1 \omega) (x - y) \times y}{|y| ^3} \d y
    \right| &\le C \int _{3 \mathcal C} |(\phi _1 \omega) (x - y)| |y| ^{-2} \d y \\
    & \lesssim \nor{\phi _1 \omega} _{\mathrm{L}^6 (\R ^3)} \left(
        \int _{3 \mathcal C} \frac1{|y| ^\frac{12}5} \d y
    \right) ^\frac56
\end{align*}
for every $x \in \R ^3$, where we may additionally bound
\begin{align*}
    \int _{3 \mathcal C} \frac1{|y| ^\frac{12}5} \d y &\le \int _{-R} ^R \int _{-R} ^R \int _{\R} \frac1{(y _1 ^2 + y _2 ^2 + y _3 ^2) ^\frac65} \d y _1 \d y _2 \d y _3 \\
    &= C \int _{-R} ^R \int _{-R} ^R \frac1{(y _2 ^2 + y _3 ^2) ^{\frac7{10}}} \d y _2 \d y _3 = C R ^\frac35.
\end{align*}
Therefore 
\begin{align*}
    \nor{v _1} _{\mathrm{L}^\infty} \lesssim \Xi R ^\frac12 V ^\frac12 \left(\frac {UR}L + \frac1{R} + L \Xi \right) (\Xi V ^\frac23 + 1).
\end{align*}
Similarly, we can define $h _1 = \zeta - v _1$, which is harmonic, and use our previous argument to get $\nor{h _1} _{\mathrm{L}^\infty (\mathcal C)} \lesssim L \Xi$, so 
\begin{align*}
    \nor{\zeta} _{\mathrm{L}^\infty (\mathcal C)} \lesssim \Xi R ^\frac12 V ^\frac12 \left(\frac {UR}L + \frac1{R} + L \Xi \right) (\Xi V ^\frac23 + 1) + L \Xi.
\end{align*}
This completes the proof of \cref{lem: oscillation-estimate}.
\qed

\subsection{Ansatz}
\label{sec: ansatz}

Suppose that for some $\varepsilon _0 > 0$, $0 \le \delta \le \frac5{12}$ to be determined later, the following holds:
\begin{align}
    \label{eqn: ansatz}
    L \Xi \le \frac{\varepsilon _0}{R} \pthf LR ^\delta \lesssim \varepsilon _0 \left(\frac {UR}L + \frac1{R}\right).
\end{align}
With such an ansatz, we have
\begin{align}
    \notag
    \nor{u - U e _1} _{\mathrm{L}^\infty (\mathcal C)} &\lesssim \varepsilon _0 \pthf LR ^{\delta - \frac12} \left(\frac {UR}L + \frac1{R} \right) \pthf LR ^{(\delta - \frac13) _+} + \varepsilon _0 \left(\frac {UR}L + \frac1{R}\right) \\
    \label{eqn: linfty-estimate}
    & \lesssim \varepsilon _0 \left(\frac {UR}L + \frac1{R} \right).
\end{align}
This is the ansatz with which we can prove Liouville theorems in the following sections.

\section{Capsules and a Covering Lemma}

We introduce a way of putting a capsule satisfying the ansatz around every point in $\R^3$, and we prove a covering lemma for this collection of capsules. To begin with, recall that \cref{lem: oscillation-estimate} controls the oscillation of $u$ in $\mathcal C$ but it uses information of $|\grad u| ^2$ in $2 \mathcal C$, which is not desirable. In the first subsection, we introduce the maximal functions, with which we can control $|\grad u| ^2$ in $2 \mathcal C$ using only the information in $\mathcal C$.

\subsection{Maximal Functions}

Recall the definition of the maximal function $\cM f$ of any $f \in \mathrm L ^1 _{\textrm{loc}} (\R ^3)$:
\begin{align*}
    \cM f (x) = \sup _{r > 0} \fint _{B _r (x)} |f (y)| \dy.
\end{align*}
We also introduce a streamwise maximal function $\cM _\Phi f$ associated with the flow map $\Phi$ of an incompressible vector field.

\begin{definition}
    Let $\Phi _s: \R ^3 \to \R ^3$ with $s \in \R$ be the flow map associated with the divergence free vector field $u$. In particular, $\Phi$ is given by the solution to the following initial value problem:
    \begin{align*}
        \frac\partial{\partial s} \Phi _s (x) &= u (\Phi _s (x)), & \Phi _0 (x) = x.
    \end{align*}
    For any continuous function $f \in C (\R ^3)$, we define the streamwise maximal function $\cM _\Phi f: \R ^3 \to \R$ by
    \begin{align}
        (\cM _\Phi f) (x) := \sup_{s > 0} \, \frac{1}{2 s} \int_{-s} ^s |(f \circ \Phi _\tau) (x)| \d \tau.
    \end{align}
\end{definition}

By continuity, $\cM _\Phi f \ge f$ in the pointwise sense. In analogy with the classical maximal function, $\cM _\Phi$ is also an operator of strong $(p, p)$-type for $p > 1$.

\begin{lemma}
    If $f \in C (\R ^3) \cap L ^p (\R ^3)$ with $1 < p < \infty$, then 
    \begin{align*}
        \| \cM _\Phi f \| _{L ^p (\R ^3)} \lesssim \| f \| _{L ^p (\R ^3)}.
    \end{align*}
\end{lemma}

\begin{proof} 
    Without loss of generality, assume $f \ge 0$. For $T > 0$, denote 
    \begin{align*}
        (\cM _{\Phi, T} f) (x) := \sup_{0 < s < T} \, \frac{1}{2 s} \int_{-s} ^s |(f \circ \Phi _\tau) (x)| \d \tau.
    \end{align*}
    Since $u$ is incompressible, $\Phi _s$ is a measure preserving map. Then for any $s' \in [-T, T]$ we have
    \begin{align*}
        \| \cM _{\Phi, T} f \| _{L ^p (\R ^3)} ^p & = \int_{\R^3} \left( \sup_{0<s<T} \frac{1}{2 s} \int_{-s}^{s} f \circ \Phi_\tau \d \tau \right)^p \dx \\
        & = \int_{\R^3} \left( \sup_{0<s<T} \frac{1}{2s} \int_{-s}^{s} f \circ \Phi_{\tau + s'} \d\tau \right)^p \dx \\
        & = \int_{\R^3} \left( \sup_{0<s<T} \frac{1}{2s} \int_{s' - s} ^{s' + s} f \circ \Phi_{\tau} \d\tau \right)^p \dx \\
        & = \int_{\R^3} \left( \sup_{0<s<T} \frac{1}{2s} \int_{s' - s} ^{s' + s} F_T(\tau,x) \d\tau \right)^p \dx,
    \end{align*}
    where we have denoted $F _T (\tau, x) := (f \circ \Phi _\tau (x))\cdot \mathbf1 _{(-2T,2T)} (\tau)$. The previous equality holds for any $s' \in [-T, T]$, so we may take an average in $s'$:
    \begin{align*}
        \| \cM _{\Phi, T} f \| _{L ^p (\R ^3)} ^p & = \frac1{2T} \int _{-T} ^T \int_{\R^3} \left( \sup_{0<s<T} \frac{1}{2s} \int_{s' - s} ^{s' + s} F _T (\tau, x) \d\tau \right)^p \dx \d s' \\
        & \le \frac1{2T} \int_{\R^3} \int _{-\infty} ^\infty \left( \sup_{s > 0} \frac{1}{2s} \int_{s'-s}^{s'+s} F_T(\tau,x) \d\tau \right)^p \d s' \dx.
    \end{align*}
    We know that the classical one-dimensional maximal function maps $L ^p$ to $L ^p$, so we get 
    \begin{align*}
        \int _{-\infty} ^\infty \left( \sup_{s > 0} \frac{1}{2s} \int_{s'-s}^{s'+s} F_T(\tau,x) \d\tau \right)^p \d s' & \le  C \int _{-\infty}^{\infty} F _T (\tau, x) ^p \d \tau \\
        & = C \int _{-2 T}^{2 T} F _T (\tau, x) ^p \d \tau.
    \end{align*}
    In the second step, we used that $F _T(\tau,x)$ is supported in $\{-2 T \le \tau \le 2 T\}$. Once again using the fact that $\Phi _t$ is measure preserving, we have
    \begin{align*}
    \| \cM _{\Phi, T} f \| _{L ^p (\R ^3)} ^p &\le \frac{C}{2T} \int_{\R^3} \int_{-2T}^{2T} F_T(\tau,x) ^p \d\tau \dx \\
    &= \frac{C}{2T} \int_{-2T}^{2T} \int_{\R^3} (f \circ \Phi_\tau)^p \dx \d\tau \\
    &= C \| f \|_{L^p(\mathbb{R}^3)}^p \cdot \frac{1}{2T} \int_{-2T}^{2T} d\tau \\
    &\le 2C \| f \|_{L^p(\mathbb{R}^3)}^p.
    \end{align*}
    Since $\| \cM _{\Phi, T} f \| _{L ^p} \le C \| f \| _{L ^p}$ for a constant $C$ independent of $T$, we may use the monotone convergence theorem to conclude $\| \cM _{\Phi} f \| _{L ^p} \le C \| f \| _{L ^p}$ as well.
\end{proof}

Recall that $u$ is a smooth vector field, so $|\grad u| ^2$ is globally Lipschitz continuous. Its maximal function
\begin{align*}
    \cM (|\grad u| ^2) (x) = \sup _{r > 0} \fint _{B _r (x)} |\grad u (x')| ^2 \dx' = \sup _{r > 0} \frac1{|B _r|} (\mathbf1 _{B _r} * |\grad u| ^2) (x),
\end{align*}
is also Lipschitz continuous with the same Lipschitz bound. Therefore, the composition of maximal functions $\cM _\Phi [\cM (|\grad u| ^2)]$ in the following lemma is well-defined.

\begin{lemma}
    \label{lem: maximal-controls-2C}
    Given $\mathcal C = \mathcal C _{R, L, e _1} (0)$, assume there exist $U, \tilde \Xi > 0$, such that 
    \begin{align}
        \label{eqn: U-and-Xi-3}
        \fint _{B _R} u \dx = U e _1, \qquad \fint _{\mathcal C} \cM _\Phi [\cM (|\grad u| ^2)] \dx = \tilde \Xi ^2.
    \end{align}
    There exist absolute constants $C > 0$ and $\varepsilon_1>0$ such that if either $L \le 2 R$ or $L \tilde \Xi \le \varepsilon _1 \frac{R}{L} \cdot U$, then 
    \begin{align}
        \label{eqn: bound-dirichlet-by-maximal}
        \fint _{2 \mathcal C} |\grad u| ^2 \dx \le C \fint _{\mathcal C} \cM _\Phi [\cM (|\grad u| ^2)] \dx = C \tilde \Xi ^2.
    \end{align}
\end{lemma}

\begin{proof}
    If $L \le 2 R$, we do not require use of the streamwise maximal function. Indeed, since $B _{R} (x) \subset 2 \mathcal C \subset B _{6 R} (x)$ for any $x \in \mathcal C$, we have
    \begin{align*}
        \fint _{2 \mathcal C} |\grad u| ^2 \dx &\lesssim \fint _{B _{6 R} (x)} |\grad u| ^2 \dx \le \cM (|\grad u| ^2) (x)  \le \cM _\Phi [\cM (|\grad u| ^2)] (x).
    \end{align*}
    As this holds for any $x \in \mathcal C$, we can take average of the right-hand side and conclude \eqref{eqn: bound-dirichlet-by-maximal}.
    
    From now on, we may assume $L > 2 R$. Applying \cref{SILLYLEMMA} twice, we have 
    \begin{align*}
        \int _{\mathcal C _{2R, L+R}} |\grad u| ^2 \dx &\le \frac CR \int _{-L + R} ^{L - R} \int _{B _{4 R}} |\grad u (x + t e _1)| ^2 \dx \dt \\
        &\le \frac CR \int _{-L + R} ^{L - R} \int _{B _{R}} \cM (|\grad u| ^2) (x + t e _1) \dx \dt \\ 
        & \le C \int _{\mathcal C _{R, L}} \cM (|\grad u| ^2) \dx
        \le C \int _{\mathcal C _{R, L}} \cM _\Phi [\cM (|\grad u| ^2)] \dx.
    \end{align*}
    Since the volume of $\mathcal C _{2R, L+R}$ is comparable to $\mathcal C _{R, L}$, we conclude 
    \begin{align*}
        \fint _{\mathcal C _{2R, L+R}} |\grad u| ^2 \dx &\le C \tilde \Xi ^2.
    \end{align*}
    The length $L+R$ is smaller than the desired $2 L$ length of $2 \mathcal C$ in Equation \eqref{eqn: bound-dirichlet-by-maximal}. To improve it, we use a bootstrap argument. 

    \paragraph{Step 1: The Dirichlet Integral Controls the Proximity of Streamlines}

    We claim that for any $L' \in [L + R, 2 L]$, if 
    \begin{align}
        \label{eqn: H1-bound-by-maximal}
        \fint _{\mathcal C _{2R, L'}} |\grad u| ^2 \dx \le C \tilde \Xi ^{2}, 
    \end{align}
    then 
    \begin{align}
        \label{eqn: sufficient-intersection}
        |\Phi _{\frac{t _1 - t _0}U} (B _R (t _0 e _1)) \cap B _R (t _1 e _1)| \ge \mu |B _R|, \qquad \forall |t _0|, |t _1| \le L' - 2 R
    \end{align}
    with $\mu = 1 - C \varepsilon _1$. 
    Indeed, the measure on the left-hand side of \eqref{eqn: sufficient-intersection} is given by 
    $$\int_{B_R(t _1 e_1)} \rho \left(\frac {t _1}U, x\right) \dx,$$
    where $\rho (s, x)$ is the solution to the following transport equation:
    $$\partial_s \rho + \textrm{div}(\rho u) =0, \quad \rho \left(\frac {t _0}U, x\right) = \mathbf{1}_{B_R(t _0 e _1)} (x).$$
    Note that $|\rho|\leq 1$ everywhere in space and time. We observe that
    \begin{align*}
        \frac{\d}{\d s} \int_{B_R(s U e_1)} \rho(s, x) \dx 
        &= \frac{\d}{\d s} \int_{B_R(0)} \rho(s, x+s Ue_1)\dx \\ 
        &=\int_{B_R(s Ue_1)} \partial_s \rho(s, x)\dx +\int_{B_R(s Ue_1)} U\partial_{x_1}\rho(s, x)\dx \\
        &= \int_{B_R(s Ue_1)}\textrm{div}(\rho(U e _1 - u))\dx \\
        &= \int_{\partial B_R(s Ue_1)}\rho(U e _1-u)\cdot n \d\sigma.
    \end{align*}
    Here $|\rho| \le 1$ and $n$ is the unit outward normal vector. By the Trace Theorem, we can estimate:
    \begin{align}
        \label{eqn: diff-ineq}
        \left|\frac{\d}{\d s}\int_{B_R(s Ue_1)} \rho(s, x)\dx \right| &\lesssim \int_{B_R(s Ue_1)} |\nabla u| \dx  + |\partial B _R| \cdot |U e _1 - \bar u (s U)|.
    \end{align}
    By \eqref{eqn: H1-bound-by-maximal}, we can apply \cref{lem: mean-velocity} to $\mathcal C _{2R, L'}$ instead of $\mathcal C _{2R, 2L}$ and deduce for all $|s U| \le L' - 2 R$ that 
    \begin{align*}
        |U e _1 - \bar u (s U)| \lesssim C L \tilde \Xi.
    \end{align*}
    
    Without loss of generality, assume $t _0 < t _1$. Integrate Equation \eqref{eqn: diff-ineq} from $s = \frac {t _0}U$ to $s = \frac{t _1} U$, we get 
    \begin{align*}
        \int_{B_R(t _1 e_1)} \rho \left(\frac {t _1}U, x\right) \dx \ge \int_{B_R(t _0 e_1)} \rho \left(\frac {t _0}U, x\right) \dx - C \int _{\frac {t _0}U} ^\frac {t _1}U \int_{B_R(s Ue_1)} |\nabla u| \dx \d s \\
        - |\partial B _R| \cdot C L \tilde \Xi \cdot \frac{t _1 - t_0}{U}.
    \end{align*}
    For the first term on the right, using the initial condition of $\rho$ we get 
    \begin{align*}
        \int_{B_R(t _0 e_1)} \rho \left(\frac {t}U, x\right) \dx = |B _R|.
    \end{align*}
    The second term is estimated using \cref{SILLYLEMMA}:
    \begin{align*}
        \int _{\frac {t _0}U} ^\frac {t _1}U \int_{B_R(s Ue_1)} |\nabla u| \dx \d s &= \frac1U \int _{t _0} ^{t _1} \int_{B_R (t e_1)} |\nabla u| \dx \dt \\
        &\le \frac {2R^3 L}U \left(\fint _{\mathcal C _{R, L'}} |\grad u| ^2  \dx\right)^{1/2} \le \frac{2 R ^3}{U} L \tilde \Xi.
    \end{align*}
    The third term is bounded by 
    \begin{align*}
        |\partial B _R| \cdot C L \tilde \Xi \cdot \frac{t _1 - t _0}{U} \lesssim R ^2 L \tilde \Xi \frac{L'}{U} \lesssim \frac{R ^2 L}{U} L \tilde \Xi.
    \end{align*}
    By assumption, $L \tilde \Xi \lesssim \varepsilon _1 \frac{R}{L} \cdot U$, so 
    \begin{align*}
        \int_{B_R(t _1 e_1)} \rho \left(\frac {t _1}U, x\right) \dx \ge |B _R| - \frac{C R ^3}{U} L \tilde \Xi - \frac{C R ^2 L}{U} L \tilde \Xi \ge |B _R| - C \varepsilon _1 R ^3.
    \end{align*}
    This completes the proof for Equation \eqref{eqn: sufficient-intersection}.

    \paragraph{Step 2: The Proximity of Streamlines controls the Dirichlet Integral}
    
    We claim if \eqref{eqn: sufficient-intersection} holds at $\mu = \frac12$, then 
    \begin{align}
        \label{eqn: H1-bound-by-maximal-2}
        \fint _{\mathcal C _{2 R, L ' + 2 R}} |\grad u| ^2 \dx \le C \tilde \Xi ^{2}
    \end{align}
    So we can improve \eqref{eqn: H1-bound-by-maximal} from length $L'$ to a longer length $L' + 2 R$. 
    To show \eqref{eqn: H1-bound-by-maximal-2}, we use \cref{SILLYLEMMA} and the definition of $\cM _\Phi$:
    \begin{align*}
        \int _{\mathcal C _{R, L, e _1}} \cM _\Phi [\cM (|\grad u| ^2)] \dx &\ge \frac1R \int _{-L + R} ^{L - R} \int _{B _R} \cM _\Phi [\cM (|\grad u| ^2)] (x + t e _1) \dx \dt \\
        &= \frac1R \int _{-L + R} ^{L - R} \int _{B _R (t e _1)} \cM _\Phi [\cM (|\grad u| ^2)] \dx \dt \\
        &\ge \frac1R \int _{-L + R} ^{L - R} \int _{B _R (t e _1)} \fint _{-\frac{4 L}U} ^{\frac{4 L}U} \cM (|\grad u| ^2) \circ \Phi _s \d s \dx \dt \\
        &\gtrsim \frac 1{RL} \int _{-L + R} ^{L - R} \int _{-4 L} ^{4 L} \int _{\Phi _\frac{t'}{U} (B _R (t e _1))} \cM (|\grad u| ^2) \dx \dt'  \dt.
    \end{align*}
    For $t' \in (-4 L, 4 L)$, $t \in (-L + R, L - R)$, we have $t _1 := t' + t \in (-4 L + L - R, 4 L - L + R) \supset (-L' + 2 R, L' - 2 R)$. So 
    \begin{align*}
        & \int _{-4 L} ^{4 L} \int _{\Phi _\frac{t'}{U} (B _R (t e _1))} \cM (|\grad u| ^2) \dx \dt' \\
        & \qquad \ge \int _{-L' + 2 R} ^{L' - 2 R} \int _{\Phi _\frac{t _1 - t}{U} (B _R (t e _1))} \cM (|\grad u| ^2) \dx \d t _1 \\
        & \qquad \ge \int _{-L' + 2 R} ^{L' - 2 R} \int _{\Phi _\frac{t _1 - t}{U} (B _R (t e _1)) \cap B _R (t _1 e _1)} \cM (|\grad u| ^2) \dx \d t _1.
    \end{align*}
    Since for each $x \in B _R (t _1 e _1)$, 
    \begin{align*}
        \mathcal M (|\grad u| ^2) (x) \gtrsim \fint _{B _{8 R} (t _1 e _1)} |\grad u| ^2 \dx,
    \end{align*}
    so we can use the assumption in Equation \eqref{eqn: sufficient-intersection} with $\mu = \frac12$, to deduce
    \begin{align*}
        \int _{\Phi _\frac{t _1 - t}{U} (B _R (t e _1)) \cap B _R (t _1 e _1)} \cM (|\grad u| ^2) \dx &\gtrsim \frac12 |B _R| \fint _{B _{8 R} (t _1 e _1)} |\grad u| ^2 \dx \\
        & \approx \int _{B _{8 R} (t _1 e _1)} |\grad u| ^2 \dx. 
    \end{align*}
    Using \cref{SILLYLEMMA} we have
    \begin{align*}
        &\int _{-L' + 2 R} ^{L' - 2 R} \int _{\Phi _\frac{t _1 - t}{U} (B _R (t e _1)) \cap B _R (t _1 e _1)} \cM (|\grad u| ^2) \dx \d t _1 \dt \\
        & \qquad \gtrsim \int _{-L' + 2 R} ^{L' - 2 R} \int _{B _{8 R} (t _1 e _1)} |\grad u| ^2 \dx \dt _1 \gtrsim R \int _{\mathcal C _{4R, L' + 2 R, e _1}} |\grad u| ^2 \dx.
    \end{align*}
    Therefore, 
    \begin{align*}
        \int _{\mathcal C _{R, L, e _1}} \cM _\Phi [\cM (|\grad u| ^2)] \dx &\gtrsim \frac1{LR} \cdot (2 L - 2 R) \cdot R \int _{\mathcal C _{4R, L' + 2 R, e _1}} |\grad u| ^2 \dx \\
        & \ge \int _{\mathcal C _{2 R, L' + 2 R, e _1}} |\grad u| ^2 \dx.
    \end{align*}
    This proves \eqref{eqn: H1-bound-by-maximal-2}, as $\mathcal C _{R, L}$ and $\mathcal C _{2R, L' + 2R}$ have comparable volumes.

    \paragraph{Step 3: Bootstrap} Combining the previous two steps, we know by choosing $\varepsilon _1 \le \frac1{2 C}$, \eqref{eqn: H1-bound-by-maximal} implies \eqref{eqn: H1-bound-by-maximal-2} for every $L' \in [L + R, 2 L]$. Therefore, after finitely many iterations, we conclude \eqref{eqn: H1-bound-by-maximal} is true for $L' = 2 L$, which proves \eqref{eqn: bound-dirichlet-by-maximal}.
\end{proof}

\subsection{Capsule Construction}

For every point $x \in \R^3$, we find a capsule that satisfies our ansatz. 

\begin{proposition}
    \label{prop: construction}
    Fix $0 < \varepsilon _0 < 1$, $0 \le \delta \le \frac5{12}$ and $\sigma \ge \delta$. Let $\varepsilon _1 \ll \varepsilon _0$. For every $x \in \R ^3$, there exists a capsule $\mathcal C _x = \mathcal C _{R (x), L (x), e (x)} (x)$ with radius $R (x)$ and length $L (x)$ pointing in the direction $e (x) \in \mathbb S ^2$ such that
    \begin{align}
        \label{eqn: equality-U-Xi}
        \fint _{B _{R (x)} (x)} u \dx &= b (x) = U (x) e (x), & \fint _{\mathcal C _x} \cM _\Phi [\cM (|\grad u| ^2)] \dx &= \tilde \Xi (x) ^2,
    \end{align}
    where $U (x)$ and $\tilde \Xi (x)$ satisfy
    \begin{align}
        \label{eqn: equality-ULR-Xi}
        0\leq U (x) &\le \frac{L (x)}{R (x) ^2} \pthf{L (x)}{R (x)} ^\sigma, & \tilde \Xi (x) &= \frac{\varepsilon _1}{L (x) R (x)} \pthf{L (x)}{R (x)} ^\delta.
    \end{align}
    The right-hand-side inequality for $U(x)$ becomes an equality when $L (x) > R (x)$. Moreover, it holds that
    \begin{align*}
        \nor{u-b(x)} _{\mathrm{L}^\infty (\mathcal C _x)} &\lesssim \frac1{R (x)}, & \text{ if } L (x) = R (x), \\
        \nor{u - b (x)} _{\mathrm{L}^\infty (\mathcal C _x)} &\lesssim \varepsilon _0 \frac{R (x)}{L (x)} \cdot U (x), & \text{ if } L (x) > R (x).
    \end{align*}
    In particular 
    \begin{align*}
        |u (x)| \lesssim \frac{L (x)}{R (x) ^2} \left(\frac{L (x)}{R (x)}\right) ^{\sigma}.
    \end{align*}
\end{proposition}
    
\begin{proof}
    We first describe an algorithm for constructing the capsule $\mathcal C _x$. To begin with, for each $x \in \R ^3$, $R > 0$, define 
    \begin{align*}
        \mathcal C _{R} (x) := \mathcal C _{R, L ^*, e ^*} (x)
    \end{align*}
    with $L ^* \ge R$ and $e ^* \in \mathbb S ^2$ determined by 
    \begin{align*}
        \fint _{B _R (x)} u \dx &= b ^* = U ^* e ^*, & 
        L ^* = \max \{U ^* R, 1\} ^\frac1{1 + \sigma} R.
    \end{align*}
    Here $U ^* = |b ^*| \ge 0$, so $U ^*$ and $e ^*$ are uniquely determined whenever the average velocity $b ^*$ is nonzero. If the average velocity is zero, $U ^* = 0$, but the definition of $e ^*$ is ambiguous. However, in this case we have $L ^* = R$, so the definition of $\mathcal C _R (x)$ is the same for whichever $e ^*$ one wishes to pick.

    From the above construction, we see that $U ^*$ and $L ^*$ change continuously in $R$ and $x$, so let us denote $U ^* =: U _R (x)$ and $L ^* =: L _R (x)$. The capsule $\mathcal C _R (x)$ also changes continuously in $R$ and $x$, because $e ^*$ is only discontinuous in the closed set $\{(x, R) \in \R ^3 \times [0, \infty): U _R (x) = 0\}$, but $\mathcal C _R (x)$ is a ball whenever $U _R (x) < \frac1R$. Therefore, $\tilde \Xi _R (x)$ defined below is also continuous in $R$ and $x$:
    \begin{align*}
        \tilde \Xi _R (x) := \left(\fint _{\mathcal C _x (R)} \cM _\Phi [\cM (|\grad u| ^2)]  \dx\right) ^\frac12.
    \end{align*}
    
    It remains to find a critical value $R > 0$ for each $x$ such that $$\tilde \Xi _R (x) L _R ^{1 - \delta} (x) R ^{1 + \delta} = \varepsilon _0.$$ For a fixed $x$, we have shown that $\tilde \Xi _R (x) L _R^{1-\delta} (x) R^{1+\delta}$ is a continuous function of $R \in [0, \infty]$. Note the following limits:
    \begin{align*}
        \lim _{R \to 0} U _R (x) &= |u (x)|, & 
        \lim _{R \to 0} L _R (x) &= 0, & 
        \lim _{R \to 0} \tilde\Xi _R (x) &= (\cM_\Phi [\cM(|\nabla u|^2)] (x)) ^{1/2}, \\
        \lim _{R \to \infty} U _R (x) &= 0, & 
        \lim _{R \to \infty} L _R (x) &= \infty, & 
        \lim _{R \to \infty} \tilde\Xi _R (x) &= 0.
    \end{align*}
    Therefore, 
    \begin{align*}
        \lim _{R \to 0} \tilde \Xi _R (x) L _R ^{1 - \delta} (x) R ^{1 + \delta} &= 0, \\
        \lim _{R \to \infty} \tilde \Xi _R (x) L _R ^{1 - \delta} (x) R ^{1 + \delta} &= \lim _{R \to \infty} \left(\fint _{\mathcal C _x (R)} \cM _\Phi [\cM (|\grad u| ^2) \dx\right) ^\frac12 L _R ^{1 - \delta} (x) R ^{1 + \delta} \\
        &\ge \lim _{R \to \infty} \left(L _R ^{1 - 2 \delta} (x) R ^{2 \delta} \int _{\mathcal C _x (R)} |\grad u| ^2 \dx\right) ^\frac12 \\
        &= \lim _{R \to \infty} \left(L _R ^{1 - 2 \delta} (x) R ^{2 \delta} \int _{\R ^3} |\grad u| ^2 \dx\right) ^\frac12 = +\infty.
    \end{align*}    
    By continuity, equality $\tilde\Xi _R L ^{1 - \delta} _R R ^{1 + \delta} = \varepsilon _0$ must be achieved at some $R ^* > 0$. We pick the smallest $R ^*$ for which the equality is achieved. We can now finally define 
    \begin{align*}
        \mathcal C _x := \mathcal C _{R ^*} (x).
    \end{align*}
    
    We remark again that even though $e$ is not uniquely defined when $U = 0$, the capsule $\mathcal C _x$ degenerates to a ball in this case, so the choice of $e$ is unimportant. Now each $x\in \R^3$ is associated with a cylinder $\mathcal C _x$. Moreover, we have chosen $R (x)$, $b (x)$, $U (x)$, $e (x)$, $L (x)$, $\tilde\Xi (x)$, with the equalities in Equation \eqref{eqn: equality-U-Xi}  holding precisely. As functions of $x$, these functions are not necessarily continuous due to the way we pick $R$, however, in the next subsection we shall show they are locally comparable.

    Let us verify the ansatz \eqref{eqn: ansatz} is satisfied with $\Xi (x) ^2 = \fint _{2 \mathcal C _x} |\grad u| ^2 d x$. The conclusions in Equation \eqref{eqn: equality-ULR-Xi} follow from our construction. If $L = R$, then 
    \begin{align*}
        \tilde \Xi (x) = \frac{\varepsilon _0}{L (x) R (x)} \pthf{L (x)}{R (x)} ^\delta = \frac{\varepsilon _1}{L (x) R (x)}.
    \end{align*}
    By \cref{lem: maximal-controls-2C}, if we choose $\varepsilon _1 \le \frac1C \varepsilon _0$ then $\Xi (x) \le \frac{\varepsilon _0}{L (x) R (x)}$, and \eqref{eqn: ansatz} is true. If $L > R$, then since $\sigma \geq \delta$
    \begin{align*}
        U (x) = \frac{L (x)}{R (x) ^2} \pthf{L (x)}{R (x)} ^\sigma \ge \frac{L (x)}{R (x) ^2} \pthf{L (x)}{R (x)} ^\delta,
    \end{align*}
    and \eqref{eqn: ansatz} is still true.

    By our choice of $L (x)$, we see that if $L (x) = R (x)$, then $U (x) R (x) \le 1$, so $U (x) \le \frac1{R (x)}$, then Equation \eqref{eqn: linfty-estimate} gives the desired estimate. If $L (x) > R (x)$, then since $\sigma \geq 0$:
    \begin{align*}
        L = (U R) ^\frac1{1 + \sigma} R \implies U = \frac1R \pthf LR ^{1 + \sigma} \ge \frac L{R ^2}.
    \end{align*}
    Equation \eqref{eqn: linfty-estimate} once again yields the desired estimate.
\end{proof}

\begin{remark}
    \label{rem: bound-below}
    Since $u$ and $\grad u$ are globally bounded from above, we can see that $R (x)$ and $L (x)$ are globally bounded from below. Equivalently we have that $1/R(x)$, $1/L(x)$ are globally bounded from above.
\end{remark}

It will be useful to distinguish the two cases $L = R$ and $L > R$. 

\begin{definition}
    We call $x \in \R ^3$ a round point if $L (x) = R (x)$, and we call $x\in \R^3$ a long point otherwise. The set of round points is denoted by $\mathcal R$, and the set of long points is denoted by $\mathcal L$.
\end{definition}

\subsection{Comparability and a Vitali Covering Lemma}

Let $x \in \mathcal L$ be a long point, then by \cref{prop: construction} we know
\begin{align*}
    \nor{u - b (x)} _{\mathrm{L}^\infty (\mathcal C _x)} \le \varepsilon _0 \frac {R (x)}{L (x)} \cdot U (x) \le \varepsilon _0 U (x).
\end{align*} 
In particular, $|u (y) - b (x)| \le \varepsilon _0 U (x)$ for every $y \in \mathcal C _x$. Therefore
\begin{align*}
    (1 - \varepsilon _0) U (x) &\le |u (y)| \le (1 + \varepsilon _0) U (x), & |u (x) - u (y)| &\le 2 \varepsilon _0 U (x).
\end{align*}
This means that $u (x) \approx u (y) \approx b (x)$ up to an order $\varepsilon _0$ relative error.

\begin{lemma}
\label{CAPSULEPROPERTY}
    Suppose $x$ and $z$ are long points such that $\mathcal C _x \cap \mathcal C _z \ne \varnothing$. If $R (z) \le 2 R (x)$, then $\mathcal C _z \subset K \mathcal C _x$ for some absolute constant $K$.
\end{lemma}

\begin{proof}
    If $x,z \in \mathcal{L}$ such that $\mathcal C _x$ and $\mathcal C _z$ intersect at some point $y$, then we may conclude that $b (x) \approx u (y) \approx b (z)$ and consequently $U (x) \approx U (z)$, $e (x) \approx e (z)$. Since $x,z\in\mathcal{L}$, we know 
    \begin{align*}
        \frac{L (x) ^{1 + \sigma}}{R (x) ^{2 + \sigma}} = U (x) \approx U (z) = \frac{L (z) ^{1 + \sigma}}{R (z) ^{2 + \sigma}} && \implies && L (z) \lesssim 2 ^{\frac{2 + \sigma}{1 + \sigma}} L (x)\lesssim L(x),
    \end{align*}
    where we have used that $\sigma\geq 0$.
    Moreover, since 
    \begin{align*}
        |b (x) - b (z)| \le |b (x) - u (y)| + |u (y) - b (z)| \le \varepsilon _0 \frac{R (x)}{L (x)} U (x) + \varepsilon _0 \frac{R (z)}{L (z)} U (z),
    \end{align*}
    we know that
    \begin{align*}
        |e (x) - e (z)| \lesssim \varepsilon _0 \left(
            \frac{R (x)}{L (x)} + \frac{R (z)}{L (z)} 
        \right).
    \end{align*}

    Without loss of generality, assume $x$ is at the origin and $e (x) = e _1$ is the first basis vector. Then $e (z) = e _z ^1 e _1 + e _z ^2 e _2 + e _z ^3 e _3$ satisfies 
    \begin{align*}
        |1 - e _z ^1| + |e _z ^2| + |e _z ^3| \lesssim \varepsilon _0 \left(
            \frac{R (x)}{L (x)} + \frac{R (z)}{L (z)}
        \right).
    \end{align*}
    Taking any $\tilde y \in \mathcal C _z$, we compute its distance from $y$. For the first coordinate, we control it directly using the diameter:
    \begin{align*}
        |y _1 - \tilde y _1| \le 2 L (z) \le K L (x).
    \end{align*}
    For the second coordinate, we use the fact that the direction $e (z)$ is close to $e _1$. Since $y \in z + t e (z) + B _{R (z)}$ and $\tilde y \in z + \tilde t e (z) + B _{R (z)}$ for some $t, \tilde t \in [- L (z) + R (z), L (z) - R (z)]$, we have 
    \begin{align*}
        |y _2 - \tilde y _2| & \lesssim |t - \tilde t| |e ^2 _z| + 2 R (z) \\
        & \le 
        2 L (z) \cdot \varepsilon _0 \left(
            \frac{R (x)}{L (x)} + \frac{R (z)}{L (z)}
        \right) + 2 R (z) \\
        &= 2 \varepsilon _0 \frac{L (z)}{L (x)} \cdot R (x) + 2 \varepsilon _0 R (z) + 2 R (z) \\ 
        &\le \varepsilon _0 C R (x) + 4 R (x) \le K  R (x),
    \end{align*}
    if $\varepsilon _0$ is chosen small. By a similar estimate for the third coordinate, we conclude that $y - \tilde y \in (K + 1) \mathcal C _x$. Because $y \in \mathcal C _x$, we have $\tilde y \in (K  + 2) \mathcal C _x$, so $\mathcal C _z \subset (K + 2) \mathcal C _x$. 
\end{proof}

We now show that the property proven in \cref{CAPSULEPROPERTY} guarantees that we have the following Vitali-type covering lemma:

\begin{lemma}
\label{COVERING}
    Let $A \subset \mathcal L$ be a set such that $\sup _{x \in A} R (x) < \infty$. Then there exists a countable collection of pairwise disjoint capsules $\{ \mathcal C _{x _i} \} _i$ such that $A \subset \bigcup _i K \mathcal C _{x _i}$ and
    \begin{align*}
        \left| \bigcup _{x \in A} \mathcal C _x \right| \le \sum _{i} |K \mathcal C _{x _i}|.
    \end{align*}
\end{lemma}
\begin{proof}
    Let $\mathcal{Q}^{(0)}:= \{\mathcal C_x\}_{x\in A}$ be the whole collection of capsules. For any $j\geq 1$, we first let $x_j$ be any point such that
    $$R(x_j) > \frac{1}{2} \sup_{\mathcal C_x \in \mathcal{Q}^{(j-1)}} R(x).$$
    Then we eliminate from the collection $\mathcal{Q}^{(j-1)}$ any capsule $\mathcal{C}_x$ that intersects $\mathcal{C}_{x_j}$ to create the collection $\mathcal{Q}^{(j)}$. This algorithm may cease after finitely many iterations if the collection $\mathcal{Q}^{(j+1)}=\varnothing$ for some $j\geq 0$, or it may never cease. In any case, we have constructed a countable collection of capsules
    $\{\mathcal{C}_{x_j}\}_{j\geq 0}$ that are pairwise disjoint by construction.
    
Suppose that $\sum _j \left|K\mathcal{C}_{x_j}\right| < \infty$, otherwise the final claim is already true. It follows that either $\{\mathcal{C}_{x_j}\}_{j\geq 0}$ is a finite collection, or it is an infinite collection with $R(x_j) \to 0$ as $j \to \infty$. In either case, every capsule $\mathcal{C}_{x}$ must be removed from $\mathcal{Q}^{(j)}$ at some iteration step. Otherwise, we would have $\mathcal{C}_{x} \in \mathcal{Q}^{(j - 1)}$ for all $j$, and the first step in the algorithm would then imply that $R(x_j)> \frac12 R(x)$ for all $j$. However, the sequence $R(x_j)$ would then not converge to zero. Now suppose $\mathcal{C}_{x} \in \mathcal{Q}^{(j - 1)} \setminus \mathcal{Q}^{(j)}$. Then we have $\mathcal{C}_{x} \cap \mathcal{C}_{x_j} \neq \varnothing$ and $R(x)< 2R(x_j)$. By \cref{CAPSULEPROPERTY} we have
\begin{align*}
    \mathcal{C}_{x} \subset K\mathcal{C}_{x_j}.
\end{align*}
Thus $\bigcup _{x\in A} \mathcal{C}_{x} \subset \bigcup _{j \geq 0}  K\mathcal{C}_{x_j}$ and we have:
\begin{align*}
    \left|
        \bigcup _{x\in A} \mathcal{C}_{x}
    \right| \le 
    \left|
        \bigcup _{j\geq 0} K\mathcal{C}_{x_j} 
    \right| \le \sum _{j \geq 0} |K\mathcal{C}_{x_j} |.
\end{align*}
This proves the covering lemma.
\end{proof}

A corollary of the covering lemma is that $\tilde\Xi \in \mathrm{L}^{2^+, \infty} (\R^3)$.

\begin{corollary}
For any $\epsilon>0$, we have that
    $\tilde\Xi$ is in $\mathrm{L}^{2+\epsilon, \infty} (\R^3)$ with bound 
    \begin{align*}
        \| \tilde\Xi \| _{\mathrm{L}^{2+\epsilon, \infty} (\R^3)} \lesssim _\epsilon \nor{\nabla u} _{\mathrm{L}^{2+\epsilon} (\R ^3)}.
    \end{align*}
\end{corollary}

\begin{proof}
    We denote the set
    \begin{align*}
        S _\alpha = \{ x \in \R ^3: \tilde\Xi (x) ^2 > \alpha \},
    \end{align*}
    then $|S _\alpha| = |S _\alpha \cap \mathcal R| + |S _\alpha \cap \mathcal L|$. Note that $\tilde\Xi (x) ^2 > \alpha$ implies $R ^2 \le \frac{\varepsilon _0}{\tilde \Xi}$ is bounded in $S _\alpha$.
    For the long points, we have by \cref{COVERING} that
    \begin{align*}
        |S _\alpha \cap \mathcal L| \le \left| \bigcup _{x \in S _\alpha \cap \mathcal L} \mathcal C _x \right| \le \sum _{i} |K \mathcal C _{x _i}|.
    \end{align*}
    In the last inequality, a Vitali covering is chosen so that the capsules $\mathcal C _{x _i}$ are pairwise disjoint, and $K  \mathcal C _{x _i}$ covers $S _\alpha \cap \mathcal L$. By Jensen's inequality, we have for any $\epsilon>0$:
    $$\fint _{\mathcal C _x} \cM _\Phi [\cM (|\grad u| ^2)] \dx\leq \left(\fint _{\mathcal C _x} \cM _\Phi [\cM (|\grad u| ^2)]^{1+\epsilon} \dx\right)^{\frac{1}{1+\epsilon}}.$$  
    Moreover, $\tilde\Xi (x _i) ^2 > \alpha$ now means 
    \begin{align*}
       & \left(\fint _{\mathcal C _{x _i}} \cM _\Phi [\cM (|\grad u| ^2)]^{1+\epsilon} \dx\right)^{\frac{1}{1+\epsilon}} > \alpha \\
       \iff & \int _{\mathcal C _{x _i}} \cM _\Phi [\cM (|\grad u| ^2)]^{1+\epsilon} \dx > \alpha^{1+\epsilon} |\mathcal C _{x _i}| \\ 
       \iff & |\mathcal C _{x _i}| \le \frac{1}{\alpha^{1+\epsilon}} \int _{\mathcal C _{x _i}} \cM _\Phi [\cM (|\grad u| ^2)]^{1+\epsilon} \dx.
    \end{align*}
    Hence 
    \begin{align*}
        |S _\alpha \cap \mathcal L| & \le \frac{C}{\alpha ^{1 + \epsilon}} \sum _i \int _{\mathcal C _{x _i}} \cM _\Phi [\cM (|\grad u| ^2)]^{1+\epsilon} \dx \\
        & \le \frac {C}{ \alpha^{1+\epsilon}} \int _{\R ^3} \cM _\Phi [\cM (|\grad u| ^2)]^{1+\epsilon} \dx \\
        & \le \frac {C}{ \alpha^{1+\epsilon}} \int _{\R ^3} |\grad u| ^{2(1+\epsilon)} \dx.
    \end{align*}
    In the last inequality, we used that $\cM _\Phi \circ \cM$ is a bounded map from $L ^{1 + \epsilon}$ to $L ^{1 + \epsilon}$.
    For the round points, we can use the usual Vitali covering for balls. Combining the inequalities we have 
    \begin{align*}
        |S _\alpha| \le \frac {C}{\alpha^{1+\epsilon}} \int _{\R ^3} |\grad u| ^{2(1+\epsilon)} \dx.
    \end{align*}
    Recall $\grad u \in L ^2 \cap L ^\infty$, so the right-hand side is finite.
    Therefore $\tilde{\Xi} ^2$ is in $\mathrm{L}^{1+\epsilon, \infty}$, so $\tilde\Xi$ is in $\mathrm{L}^{2(1+\epsilon), \infty}$.
    Replacing $\epsilon$ with $\epsilon/2$ in the course of the proof gets us our claimed result.
\end{proof}

From this, we deduce that 

\begin{proposition}
\label{BESTPROP}
    If $L / R$ is bounded then $u$ is trivial.
\end{proposition}

\begin{proof}
    If $L / R$ is bounded, then $L$ and $R$ are comparable. Then by \cref{prop: construction}
    \begin{align*}
        |u (x)| \le \frac{L (x)}{R (x) ^2} \pthf{L (x)}{R (x)} ^\sigma \approx \frac C {R (x)} \approx \tilde\Xi (x) ^\frac12,
    \end{align*}
    Therefore $u \in \mathrm{L}^{4^+, \infty} (\R ^3)$. Together with $u \in \mathrm{L}^\infty$, we know that $u \in \mathrm{L}^\frac92$ and thus $u \equiv 0$ by Galdi's result.
\end{proof}

\section{Conditional Liouville Theorems}

Here we give some conditions under which $u$ is trivial, proving the theorems claimed in our introduction. 

\subsection{Bounded Line Integral Hypothesis}

The stretching ratio $L/R$ is comparable in size to a certain line integral, as the following lemma will demonstrate. 
\begin{lemma}
\label{LINEINTEGRAL}
    Let $x \in \mathcal L$ be a long point. Then 
    \begin{align}
        \label{eqn: line-integral-val}
        \int _{x - L (x) e (x)} ^{x + L (x) e (x)} u \cdot \d \ell \approx L (x) U (x).
    \end{align}
    Here $L (x)$, $U (x)$ and $e (x)$ are constructed in \cref{prop: construction}.
\end{lemma}

\begin{proof}
    According to \cref{prop: construction}, we know that in $\mathcal C _x$ we have $u (y) = U (x) e (x) + \zeta (y)$ with $|\zeta| \le C \varepsilon _0 U (x) < \frac12 U (x)$. We thus have 
    \begin{align*}
        \left|\int _{x - L (x) e (x)} ^{x + L (x) e (x)} \zeta \cdot \d \ell\right| &\le 2 L (x) \cdot \frac12 U (x) \\
        &= \frac12 \int _{x - L (x) e (x)} ^{x + L (x) e (x)} U (x) e (x) \cdot \d \ell.
    \end{align*}
    We can thus bound the line integral from above and below by:
    \begin{align*}
        \frac12 \int _{x - L (x) e (x)} ^{x + L (x) e (x)} U (x) e (x) \cdot \d \ell \le \int _{x - L (x) e (x)} ^{x + L (x) e (x)} u \cdot \d \ell \le \frac32 \int _{x - L (x) e (x)} ^{x + L (x) e (x)} U (x) e (x) \cdot \d \ell.
    \end{align*}
    The upper and lower bounds are both of order $L (x) U (x)$, which is exactly the right-hand side of Equation \eqref{eqn: line-integral-val}.
\end{proof}

Suppose that the line integral of $u$ is globally bounded: 
\begin{align*}
    \int _{x _0} ^x u \cdot \d \ell \le C, \qquad \forall x _0, x \in \R ^3,
\end{align*}
then the stretching ratio $L U \approx \pthf LR ^{2 + \sigma}$ is bounded by \cref{LINEINTEGRAL}. \cref{BESTPROP} would then imply that the solution is trivial. Since there is still a gap between $\mathrm{L}^{4 ^+, \infty}$ and $\mathrm{L}^\frac92$, we can relax the boundedness assumption and allow an assumption like Equation \eqref{eqn: second-thm} in \cref{SECONDTHM}.

\begin{proof}[Proof of \cref{SECONDTHM}]
    Combining the assumption in Equation \eqref{eqn: second-thm} with Equation \eqref{eqn: line-integral-val}, we know that at any $x \in \mathcal L$: 
    \begin{align*}
        L U \approx \left(\frac LR \right) ^{2 + \sigma} \lesssim L ^\beta. 
    \end{align*}
    As discussed in \cref{rem: bound-below}, $L$ is bounded from below, so the above inequality is also true for round points. Now by \cref{prop: construction}
    \begin{align*}
        |u| \lesssim \frac1L \left(\frac LR\right) ^{2 + \sigma} 
        \implies \frac{|u| ^p}{\tilde\Xi ^2} & \lesssim L ^{2 - p} R ^2 \left(\frac LR\right) ^{(2 + \sigma) p - 2 \delta} \\
        & = L ^{4 - p} \left(\frac LR\right) ^{(2 + \sigma) p - 2 \delta - 2} \\
        & \lesssim L ^{4 - p + \beta (p - \frac{2 (\delta + 1)}{2 + \sigma})}.
    \end{align*}
    If the power of $L$ satisfies
    \begin{align*}
        4 - p + \beta \left(p - \frac{2 (\delta + 1)}{2 + \sigma}\right) = 0,
    \end{align*}
    then $|u| ^p / \tilde\Xi ^2$ is bounded. Because $\tilde\Xi ^2 \in \mathrm{L}^{1+, \infty}$, we conclude that $u \in \mathrm{L}^{p+, \infty}$. We get the optimal $p = \frac{4 - 34 \beta / 29}{1 - \beta}$ when we set $\sigma = \frac5{12}$ and $\delta = \frac5{12}$, If $\beta < \frac{29}{193}$ then $p < \frac92$, and $u \in \mathrm{L}^{p, \infty} \cap \mathrm{L}^\infty$ implies $u \in \mathrm{L}^\frac92$. Galdi's criterion then asserts that $u \equiv 0$. 
\end{proof}

\subsection{Mean Oscillation Hypothesis}

Previous results showed that if $u \in \BMO ^{-1}$, then $u$ is trivial. We make some refinements in this direction. Suppose $u = \curl \psi$ with $\psi \in \mathrm{L}^1 _{loc}$.

\begin{lemma}
    \label{lem: bmo}
    Let $x \in \mathcal L$ be a long point. Then 
    \begin{align}
        \label{eqn: stream-integral-val}
        \int _{B _{R (x)} (x)} (\psi (y) - (\psi) _{B _{R (x)}}) \cdot (e (x) \times (y - x)) \d y \approx U (x) R (x) ^5.
    \end{align}
    Here $L (x)$, $R (x)$ and $e (x)$ are constructed in \cref{prop: construction}.
\end{lemma}

\begin{proof}
    For simplicity, assume $x = 0$ is the origin and $e (x) = e _1$, and let us abbreviate $R = R (x)$, $L = L (x)$, and $U = U (x)$. We observe that 
    \begin{align*}
    \int _{B _R} (\psi (y) - (\psi) _{B _{R}}) \cdot (e _1 \cross y) \d y &= \int _{B _R} (\psi (y) - (\psi) _{B _{R}}) \cdot \curl (y _1 y) \dx \\
    &= \int _{B _R} \curl \psi \cdot y _1 y \d y + \int _{\partial B _R} (y _1 y) \cross \psi \cdot n \d \sigma \\
    &= \int _{B _R} u \cdot y _1 y \d y.
\end{align*}
We used that $y$ is parallel to $n$ on $\partial B _R$. We want to use the fact that $u \approx U e _1$ with small relative error. Note that 
\begin{align*}
    \int _{B _R} U e _1 \cdot y _1 y \d y = U \int _{B _R} y _1 ^2 \d y = \frac U3 \int _{B _R} |y| ^2 \d y = \frac{U}{15} R ^5.
\end{align*}
By \cref{prop: construction}, we know 
\begin{align*}
    \left|\int _{B _R} (u (y) - U e _1) \cdot y _1 y \d y\right| \le \varepsilon _0 U (x) \int _{B _R} |y| ^2 \d y \le \frac{U}{30} R ^5.
\end{align*}
Therefore, we have 
\begin{align*}
    \int _{B _R} u \cdot y _1 y \d y \approx U R ^5. %
\end{align*}
This finishes the proof of the lemma.
\end{proof}

Write $\bar \psi = \fint _{B _R} \psi$. Notice that for any $s \ge 1$:
\begin{align}
    \label{eqn: control-mean-oscillation}
    \int _{B _R} (\psi - \bar \psi) \cdot (e _1 \cross y) \d y \le R ^4 \left(\fint _{B _R} |\psi - \bar \psi|^s \dx\right)^{1/s}.
\end{align}
If $\nor{\psi} _{\BMO} \le C$, then $U R \approx \left(\frac {L}{R} \right) ^{1 + \sigma} \lesssim \nor{\psi} _{\BMO}$, and we conclude that the ratio $L / R$ is globally bounded, so $u\equiv 0$ by \cref{BESTPROP}. Again, there is some room for improvement.

\begin{proof}[Proof of \cref{FIRSTTHM}]
    Combining Equation \eqref{eqn: stream-integral-val}-\eqref{eqn: control-mean-oscillation} and the assumption in Equation \eqref{eqn: first-thm}, we know that for $x \in \mathcal L$ with $R (x) > 1$ it holds that 
    \begin{align*}
        \pthf LR ^{1 + \sigma} \lesssim R ^\alpha .
    \end{align*}
    The set $R (x) \le 1$ is a bounded set, and $R$ is bounded from below, so the above holds for all $x \in \R ^3$. Again by \cref{prop: construction}
    \begin{align*}
        |u| \lesssim \frac1R \left(\frac LR\right) ^{1 + \sigma} 
        \implies \frac{|u| ^p}{\tilde\Xi ^2} & \lesssim L ^2 R ^{2 - p} \left(\frac LR\right) ^{(1 + \sigma) p - 2 \delta} \\
        & = R ^{4 - p} \left(\frac LR\right) ^{(1 + \sigma) p - 2 \delta + 2} \\
        & \lesssim R ^{4 - p + \alpha \cdot (p - \frac{2 (\delta - 1)}{1 + \sigma})}.
    \end{align*}
    If the power of $R$ satisfies
    \begin{align*}
        4 - p + \alpha \left(p + \frac{2 (1 - \delta)}{1 + \sigma}\right) = 0,
    \end{align*}
    then $|u| ^p/\tilde\Xi^2$ is bounded, and we conclude $u \in \mathrm{L}^{p ^+, \infty}$. By sending $\sigma \to +\infty$, we get any $p > \frac{4}{1 - \alpha}$. If $\alpha < \frac19$ then we can find some $p < \frac92$, and again $u \equiv 0$. 
\end{proof}

\subsection{Alternative Proofs}

In the previous proof, we can see our construction in \cref{prop: construction} and \cref{BESTPROP}, $u$ is in $\mathrm L ^{4 ^+, \infty} (\mathcal R)$ for the round points, and additional assumptions added in \cref{FIRSTTHM} and \ref{SECONDTHM} are used to show $\mathrm L ^{p, \infty} (\mathcal L)$ for the long points. We can also use an alternative ansatz to weaken what we can obtain for the round points, and strengthen the conclusions for the long points. We briefly illustrate the ideas here. 

\subsubsection{New Ansatz and Construction}

In \cref{sec: ansatz}, instead of assuming \eqref{eqn: ansatz}, we replace it by the following alternative assumption for $\lambda \ge 1$:
\begin{align}
    \label{eqn: ansatz2}
    L \Xi = \frac{\varepsilon _0}{R ^{\lambda}}.
\end{align}
This would lead to the following bounds instead of \eqref{eqn: linfty-estimate}, provided $R \ge 1$:
\begin{align*}
    \nor{u - U e _1} _{\mathrm{L}^\infty (\mathcal C)} \lesssim \Xi R ^\frac12 V ^\frac12 \left(\frac {UR}L + \frac1{R} \right) + L \Xi \lesssim \varepsilon _0 \left(\frac{UR}{L} + \frac1{R ^{\lambda}}\right).
\end{align*}
In the proof of \cref{prop: construction}, instead of setting 
$
    L ^* = \max \{U ^* R, 1\} ^\frac1{1 + \sigma} R
$,
we fix some $\gamma \in [0, \lambda]$ and set
\begin{align*}
    L ^* = \max \{U ^* R^\gamma, 1\} R.    
\end{align*}
We would end up with $U (x)$ satisfying 
\begin{align*}
    U (x) \le \frac{L (x)}{R (x) ^{1 + \gamma}}
\end{align*}
rather than \eqref{eqn: equality-ULR-Xi}. Again, equality is achieved for the long points. The following conclusions will hold:
\begin{align*}
    \nor{u - b(x)} _{\mathrm{L}^\infty (\mathcal C _x)} &\lesssim \varepsilon_0 \frac1{R (x) ^\gamma}, & \text{ if } L (x) = R (x), \\
    \nor{u - b (x)} _{\mathrm{L}^\infty (\mathcal C _x)} &\lesssim \varepsilon _0 \frac{R (x)}{L (x)} \cdot U (x), & \text{ if } L (x) > R (x), \\
    |u (x)| &\lesssim \frac{L (x)}{R (x) ^{1 + \gamma}}.
\end{align*}

\subsubsection{New Proof of Theorem 1.1}

By \cref{lem: bmo}, we conclude for a long point $x$ it holds that
\begin{align*}
    \int _{B _{R (x)} (x)} (\psi (y) - (\psi) _{B _{R (x)}}) \cdot (e (x) \times (y - x)) \d y \approx U R ^5 \approx L R ^{4 - \gamma}.
\end{align*}
Therefore 
\begin{align*}
    L R ^{4 - \gamma} \lesssim R ^4 \left(\fint _{B _R} |\psi - \bar \psi|^s \dx\right)^{1/s} \lesssim R ^{4 + \alpha} \implies L \lesssim R ^{\alpha + \gamma}.
\end{align*}
By setting $\gamma = 1 - \alpha$, we know $L / R$ is globally bounded. Therefore 
\begin{align*}
    |u| \lesssim \frac L{R ^{1 + \gamma}} \approx R ^{-\gamma}, \Xi \approx \frac{\varepsilon _0}{L R ^\lambda} \approx R ^{-1-\lambda} \implies \frac{|u| ^p}{\tilde\Xi ^2} \lesssim R ^{-p \gamma + 2 + 2 \lambda}.
\end{align*}
Setting $\lambda = 1$, we have the following consequence in the region $\{R \ge 1\}$:
\begin{align*}
    |u| ^p \lesssim \tilde\Xi ^2 \impliedby -p \gamma + 2 + 2 \lambda \le 0 \iff p (1 - \alpha) \ge 4 \iff p \ge \frac4{1 - \alpha}.
\end{align*}

\subsubsection{New Proof of Theorem 1.2}

By \cref{LINEINTEGRAL}, we conclude for a long point $x$ it holds that 
\begin{align*}
    \int _{x - L (x) e (x)} ^{x + L (x) e (x)} u \cdot \d \ell \approx L (x) U (x) \approx \frac{L (x) ^2}{R (x) ^{1 + \gamma}}.
\end{align*}
Therefore 
\begin{align*}
    \frac{L ^2}{R ^{1 + \gamma}} \lesssim L ^\beta \implies L ^{2 - \beta} \lesssim R ^{1 + \gamma}.
\end{align*}
Setting $\gamma = 1 - \beta$, we know $L / R$ is globally bounded. Proceeding as in the proof for Theorem \ref{FIRSTTHM}, we show $u \in \mathrm{L} ^{p ^+, \infty}$ with $p \ge \frac4{1 - \beta}$. Unfortunately, this method cannot cover the entire $p \ge \frac{4 - 34\beta/29}{1 - \beta}$ range.

\section*{Acknowledgments}

We are grateful to Camillo De Lellis for valuable discussions on the Navier--Stokes Equations. MPC acknowledges the support of the National Science Foundation in the form of an NSF Graduate Research Fellowship and under Grant No. DMS-2350252. JY acknowledges the support by the National Science
Foundation under Grant No. DMS-1926686.

\bibliographystyle{abbrv}
\bibliography{reference}%

\begin{appendix}

\renewcommand{\nu}{}

\section{Estimates for the Poisson Equation with Drift}
\label{app}

In this section, we discuss the Poisson equation with constant drift $b = U e _1$:
\begin{align}
    \label{eqn: Oseen}
    b \cdot \grad \theta - \nu \La \theta = f + \div g.
\end{align} 
We find the fundamental solution and derive $\mathrm{L}^p$ estimates.

\subsection{The Fundamental Solution}

We first find the fundamental solution to Equation \eqref{eqn: Oseen} in three dimensions.
\begin{lemma}
    The fundamental solution $\Gamma$ that solves $b \cdot \grad \Gamma - \nu \La \Gamma = \delta _0$ is
    \begin{align*}
        \Gamma (x) = \frac{1}{4\pi \nu r} \e ^{-\lambda (r - x _1)}, \qquad \lambda = \frac{U}{2\nu}.
    \end{align*}
    Here $x = (x _1, x _2, x _3)$, $b = U e _1$, $r = |x|$, and $U\geq 0$.
\end{lemma}

\begin{proof}
    
    To verify its validity, we first compute away from $r = 0$ the gradient $\grad \Gamma = \Gamma \grad \log \Gamma$ with
    \begin{align*}
        \grad \log \Gamma &= - \grad \log r - \lambda \grad (r - x _1) = -\frac{\hat x}{r} - \lambda (\hat x - e _1) = -\frac1r \hat x - \lambda \hat x + \lambda e _1.        
    \end{align*}
    Here we used the notation $\hat x = x / r$. Next we compute the Laplacian:
    \begin{align*}
        \La \Gamma &= \div \grad \Gamma = 
        \div (\Gamma \grad \log \Gamma) = \Gamma |\grad \log \Gamma| ^2 + \Gamma \La \log \Gamma,
    \end{align*}
    where 
    \begin{align*}
        \La \log \Gamma = \div \left(
            -\frac1r \hat x - \lambda \hat x
        \right) = - \left(\frac1r + \lambda\right) \frac2r + \frac1{r ^2} = -\left(\frac1r + \lambda\right) ^2 + \lambda ^2.
    \end{align*}
    Substituting this into $\La \Gamma$, we get
    \begin{align*}
        |\grad \log \Gamma| ^2 + \La \log \Gamma &= \left|
            \left(\frac1r + \lambda\right) \hat x - \lambda e _1
        \right| ^2 - \left(\frac1r + \lambda\right) ^2 + \lambda ^2 \\
        &= -2 \lambda e _1 \cdot \left(
            \left(\frac1r + \lambda\right) \hat x - \lambda e _1
        \right) = 2 \lambda e _1 \cdot \grad \log \Gamma.
    \end{align*}
    Hence, $\nu \La \Gamma = (2 \nu \lambda e _1 \cdot \grad \log \Gamma) \Gamma = U e _1 \cdot \grad \Gamma = b \cdot \grad \Gamma$.
    
    We have verified that $b \cdot \grad \Gamma - \nu \La \Gamma = 0$ in the open set $\R^3 \setminus \{0\}$. It remains to check the behavior near $r = 0$. Taking a smooth test function $\varphi$ supported in $B _r$, we compute 
    \begin{align*}
        \int _{B _r} (b \cdot \grad \Gamma - \nu \La \Gamma) \varphi \dx &= \int _{B _r} (-b \cdot \grad \varphi - \nu \La \varphi) \Gamma \dx \\
        & \qquad + \int _{\partial B _r} (\Gamma b - \varphi\grad \Gamma+\Gamma\nabla \varphi) \cdot  n\d \sigma .
    \end{align*}
    Fixing $\varphi$, the left-hand side is independent on the value of $r$, so we may let $r \to 0$. We notice that $\Gamma$ has the following trivial bound:
    \begin{align*}
        |\Gamma (x)| \leq \frac{1}{4\pi r}.
    \end{align*}
    All the integrals on the right-hand-side except the one with $\grad \Gamma$ will vanish as $r \to 0$. Moreover, 
    \begin{align*}
        \int _{\partial B _r} \grad \Gamma \cdot \varphi n \d \sigma = \int _{\partial B _r} \Gamma \grad \log \Gamma \cdot \varphi n \d \sigma = -\int _{\partial B _r} \Gamma \left(\frac1r \hat x + \lambda \hat x - \lambda e _1\right) \cdot \varphi n \d \sigma.
    \end{align*} 
    Only the term with $\frac1r$ will survive in the limit $r\to 0$. Note that $n = \hat x$, so we have 
    \begin{align*}
        \lim _{r \to 0} \int _{\partial B _r} (-\nu \grad \Gamma) \cdot \varphi n \d \sigma = \lim _{r \to 0} \int _{\partial B _r} \frac{\nu \Gamma}r \varphi \d \sigma =\varphi (0).
    \end{align*}
    This finishes the proof. 
\end{proof}

From the form of the fundamental solution, we see immediately that $|\Gamma| \le \frac1r$ and $\Gamma \in \mathrm{L}^{3, \infty}$.  We would like to establish an upper bound for $|\grad \Gamma|$ that is uniform in the size of $\lambda$. Note that 
\begin{align*}
    |r (\hat x - e _1)| ^2 = (x _1 - r) ^2 + x _2 ^2 + x _3 ^2 = (r - x _1) ^2 + r ^2 - x _1 ^2 = 2 r (r - x _1).
\end{align*}
Therefore 
\begin{align*}
    |\grad \log \Gamma| \le \frac1r + \lambda \sqrt{\frac{2 (r - x _1)}r} \le (\lambda (r - x _1) + 1) \sqrt{\frac{2}{r (r - x _1)}}.
\end{align*}
We used the inequality $0 \le r - x _1 \le 2 r$. It follows that
\begin{align*}
    |\grad \Gamma| \le \frac1{4 \pi \nu r} \e ^{-\lambda (r - x _1)} (\lambda (r - x _1) + 1) \sqrt{\frac{2}{r (r - x _1)}} \le C r ^{-\frac32} (r - x _1) ^{-\frac12}.
\end{align*}
We used that $\e ^{-\lambda (r - x _1)} (1 + \lambda (r - x _1))$ is bounded uniformly in $x _1$, $r$, and $\lambda$.

\begin{lemma}
    $\grad \Gamma$ is bounded in $\mathrm{L}^{\frac32, \infty} _{x _1} \mathrm{L}^\frac32 _{x _2, x _3}$ uniformly in $\lambda$.
\end{lemma}

\begin{proof}
    By direct integration:
    \begin{align*}
        \int _{\R ^2} &\left[r ^{-\frac32} (r - x _1) ^{-\frac12}\right] ^\frac32 \dx _2 \dx _3 \\ &= 2 \pi \int _0 ^\infty (\rho ^2 + x _1 ^2) ^{-\frac98} \left[(\rho ^2 + x _1 ^2) ^\frac12 - x _1 \right] ^{-\frac34} \rho \d \rho \leq \frac{4}{|x_1|}.
    \end{align*}
    Since $\frac1{x _1} \in \mathrm{L}^{1, \infty} _{x _1}$, we conclude that $\grad \Gamma \in \mathrm{L}^{\frac32, \infty} _{x _1} \mathrm{L}^\frac32 _{x _2, x _3}$.
\end{proof}

\subsection{Local \texorpdfstring{$\mathrm{L}^p$}{L p} Estimates}

We have found $\Gamma \in \mathrm{L}^{3, \infty}$ and $\nabla \Gamma \in \mathrm{L}^{\frac32, \infty} _{x _1} \mathrm{L}^\frac32 _{x _2, x _3}$, with bounds uniform in the size of the drift. Therefore, if $f \in \mathrm{L}^p (\R ^3)$, $1 < p < \frac32$, it holds that
\begin{align*}
    \nor{\Gamma * f} _{\mathrm{L}^r} \lesssim \nor{f} _{\mathrm{L}^p}
\end{align*}
with $\frac1r = \frac1p - \frac23$. If $g \in \mathrm{L}^q (\R ^3)$, $1 < q < 3$, then 
\begin{align*}
    \nor{\Gamma * \div g} _{\mathrm{L}^r} \lesssim \nor{g} _{\mathrm{L}^q}
\end{align*}
with $\frac1r = \frac1q - \frac13$. A proof of these convolution inequalities is given in \cite{ON}.

Recall $b = U e _1$. Suppose $\theta$ solves Equation \eqref{eqn: Oseen} in a \textit{capsule} $\mathcal C$ of radius $R$, length $L \ge R$.
We want to derive a local estimate on $\theta$.
\begin{lemma}
    \label{lem: local-estimate}
    If $\theta \in \mathrm{L}^q (\mathcal C)$ solves Equation \eqref{eqn: Oseen} in $\mathcal C$ with $f, g \in \mathrm{L}^q (\mathcal C)$ for some $1 < q < 3$, then
    \begin{align*}
        \nor{\theta} _{\mathrm{L}^r (\frac12 \mathcal C)} \lesssim 
        R \nor f _{\mathrm{L}^q (\mathcal C)} + \nor{g} _{\mathrm{L}^q (\mathcal C)} + \left(\frac {UR}L + \frac1R \right) \nor {\theta} _{\mathrm{L}^q (\mathcal C)} := \operatorname{RHS},
    \end{align*}
    where $\frac1r = \frac1q - \frac13$.
\end{lemma}

\begin{proof}
    
    Fix a non-negative cut-off function $\phi \in C _c ^\infty (\mathcal C)$ with $\phi \equiv 1$ inside $\frac12 \mathcal C$, satisfying bounds as in Equation \eqref{eqn: grad-phi}.
    Then 
    \begin{align*}
        b \cdot \grad (\phi \theta) - \La (\phi \theta) = \phi f + \div (\phi g) + \theta b \cdot \grad \phi - 2 \div (\theta \grad \phi) + \theta \La \phi - g \cdot \grad \phi.
    \end{align*}
    Denote 
    \begin{align*}
        \tilde f &= \phi f + \theta b \cdot \grad \phi + \theta \La \phi - g \cdot \grad \phi, &
        \tilde g &= \phi g - 2 \theta \grad \phi.
    \end{align*}
    Then $\phi \theta$ solves 
    \begin{align*}
        b \cdot \grad (\phi \theta) - \La (\phi \theta) = \tilde f + \div \tilde g.
    \end{align*}
The function $\phi\theta$ is the unique compactly supported solution, so $\phi \theta = \Gamma * (\tilde f + \div \tilde g)$. First, for the forcing term in divergence form, we note that 
    \begin{align*}
        \nor{\tilde g} _{\mathrm{L}^q (\mathcal C)} \lesssim \nor{g} _{\mathrm{L}^q (\mathcal C)} + \frac1R \nor{\theta} _{\mathrm{L}^q (\mathcal C)} \le \text{RHS}.
    \end{align*}
    Therefore, 
    \begin{align*}
        \nor{\Gamma * \div \tilde g} _{\mathrm{L}^r (\mathcal C)} \lesssim \nor{\tilde g} _{\mathrm{L}^q (\mathcal C)} \lesssim \text{RHS}.
    \end{align*}
    For the forcing term not in divergence form, we notice that 
    \begin{align*}
        \|{\tilde f}\| _{\mathrm{L}^q (\R ^3)} &\lesssim \nor f _{\mathrm{L}^q (\mathcal C)} + \frac1R \nor{g} _{\mathrm{L}^q (\mathcal C)} + \left(\frac UL + \frac1{R ^2}\right) \nor {\theta} _{\mathrm{L}^q (\mathcal C)}.
    \end{align*}
    Here we used that $|b \cdot \grad \phi| \lesssim \frac UL$. %
    Moreover, we have the elementary bound:
    \begin{align*}
        \nor{\Gamma} _{\mathrm{L}^\frac32 (3 \mathcal C)} ^\frac32 \lesssim \int _{3 \mathcal C} \frac1{r ^\frac32} \d x \le C R ^\frac32,
    \end{align*}
    where the constant $C$ does not depend on $L$. Suppose $\tilde{\Gamma}= \phi_1 \Gamma$, where $\phi_1$ is a smooth function supported on $2\mathcal{C}$ and vanishing off $3\mathcal{C}$. Since $\tilde f$ is supported in $\mathcal C$, we have that $(\tilde{\Gamma}\ast \tilde{f})(x) = (\Gamma\ast\tilde{f})(x)$ for all $x\in \mathcal{C}$. In addition, we have by Young's convolution inequality:
    \begin{align*}
        \nor{\Gamma * \tilde f} _{\mathrm{L}^r (\mathcal C)}= \nor{\tilde{\Gamma} * \tilde f} _{\mathrm{L}^r (\mathcal C)} &\le \nor{\tilde{\Gamma} * \tilde f} _{\mathrm{L}^r (\R^3)}\leq \nor{\Gamma} _{\mathrm{L}^\frac32 (3 \mathcal C)} \|\tilde f\| _{\mathrm{L}^q (\mathcal C)} \\
        &\lesssim R \| f \| _{\mathrm{L}^q (\mathcal C)} + \nor{g} _{\mathrm{L}^q (\mathcal C)} + \left(\frac {UR}L + \frac1R \right) \nor {\theta} _{\mathrm{L}^q (\mathcal C)}.
    \end{align*}
    Since $\phi \equiv 1$ in $\frac12 \mathcal C$, we have completed the proof.
\end{proof}

\end{appendix}

\end{document}